\documentclass[12pt,reqno]{amsart}

\usepackage[margin=1.2in]{geometry}
\usepackage{amsmath}
\usepackage{amssymb}
\usepackage{amsthm}
\usepackage{dsfont}
\usepackage{microtype}

\usepackage[cal=euler]{mathalpha}
\usepackage{charter}

\usepackage{xcolor}
\usepackage[backref=page]{hyperref}
\hypersetup{
    colorlinks=true,
    linkcolor=magenta!80!black,
    citecolor=magenta!80!black,
    urlcolor=magenta,
}

\usepackage{graphicx} 
\graphicspath{{./figures/}}

\newcommand{\bbE}{\mathbb{E}}

\newcommand{\bbP}{\mathbb{P}}

\newcommand{\bbZ}{\mathbb{Z}}

\newcommand{\bbT}{\mathbb{T}}

\newcommand{\be}{\mathbf{e}}

\newcommand{\bone}{\mathbf{1}}

\newcommand{\cF}{\mathcal{F}}

\newcommand{\cS}{\mathcal{S}}

\newcommand{\Unif}{\mathsf{Uniform}}
\newcommand{\Ber}{\mathsf{Ber}}
\newcommand{\Rad}{\mathsf{Rademacher}}

\newcommand{\ind}{\mathds{1}}

\newcommand{\Cay}{\mathrm{Cay}}

\newcommand{\convas}{\xrightarrow{\text{a.s.}}}

\newcommand{\convd}{\xrightarrow{d}}

\let\tilde\widetilde

\theoremstyle{plain}
\newtheorem{theorem}{Theorem}[section]
\newtheorem{corollary}{Corollary}[section]
\newtheorem{proposition}{Proposition}[section]
\newtheorem{lemma}{Lemma}[section]
\newtheorem{conjecture}{Open problem}[section]

\theoremstyle{definition}

\theoremstyle{remark}
\newtheorem{remark}{Remark}[section]

\allowdisplaybreaks
\title[Elephant random walks on infinite Cayley trees]{Elephant random walks on infinite Cayley trees}
\author[S. S. Mukherjee]{Soumendu Sundar Mukherjee}
\address{
    Statistics and Mathematics Unit \\
    Indian Statistical Institute \\
    203 B.T. Road, Kolkata 700108 \\
    West Bengal, India.
}
\email{ssmukherjee@isical.ac.in}

\begin{document}

\begin{abstract}
We introduce a generalisation of Sch\"{u}tz and Trimper's elephant random walk to finitely generated groups. We focus on the simplest non-abelian setting, i.e. groups whose Cayley graphs are homogeneous trees of degree $d \ge 3$. We show that the asymptotic speed of the walk does not depend on the memory parameter $p \in [0, 1)$ and equals $\frac{d - 2}{d}$, the asymptotic speed of simple random walk on these graphs. We also establish upper bounds on the rate of convergence to the limiting speed. These upper bounds depend on $p$ and exhibit a phase transition at the critical value $p_d = \frac{d + 1}{2d}$. Numerical experiments suggest that these upper bounds are tight. Along the way, we also obtain estimates on the return probability.

\end{abstract}

\keywords{Elephant random walk; step-reinforced random walk; Bethe lattice; Cayley trees; P\'{o}lya's urn}
\subjclass[2020]{60G50, 82C41, 60K99, 60G42}

\maketitle
\thispagestyle{empty}

\section{Introduction}
The Elephant Random Walk (ERW) on $\bbZ$ was introduced by Sch\"{u}tz and Trimper \cite{schutz2004elephants} to model the effect of memory in random walks and is known to exhibit anomalous diffusion. In the ERW, whose name alludes to the long memory attributed to elephants, the walker starts at the origin and takes their first step randomly to the left or right with equal probability. After this, at each time step $n > 1$, the walker chooses a previous epoch $k$ uniformly at random from $\{1, 2, \ldots, n - 1\}$ and repeats the step taken at time $k$ with probability $p$ (which is called the \emph{memory parameter}), or takes the opposite step with probability $1-p$. This reinforcement of past steps creates a long-range memory effect, influencing the walk’s future behavior in non-trivial ways.

\cite{baur2016elephant} observed a connection between the ERW and urn models, utilising this to establish the long-term behavior of ERW across all memory regimes. Notably, the ERW undergoes a phase transition at $p = 3/4$; below this threshold, the walk is diffusive, while above it, the walk becomes superdiffusive. Some of these results were also obtained in \cite{coletti2017central}, directly using martingales. Further refinements (e.g., a quadratic strong law and a law of the iterated logarithm in the regime $p \le 3/4$) were obtained by \cite{bercu2017martingale} using martingale techniques. \cite{coletti2021asymptotic} established that the ERW is recurrent for $p \le 3/4$, and transient otherwise.

\cite{bercu2019multi} and \cite{bertenghi2022functional} studied the extension of the ERW to $\bbZ^d$, also called the multi-dimensional elephant random walk (MERW), and established various limit theorems. Recent studies, such as \cite{qin2025recurrence} and \cite{curien2024recurrence}, have investigated recurrence and transience in MERW. \cite{qin2025recurrence} established that for $d \ge 3$, the walk is transient, whereas for $d = 2$, the walk becomes transient for $p \ge 5/8$. Recurrence in the sub-critical regime $p < 5/8$ for $d = 2$ was also established by \cite{curien2024recurrence} by a rough comparison with the standard plane random walk.

Beyond the classical lattice setting, random walks have been extensively studied on a variety of infinite graphs. Indeed, a fruitful direction of modern research in both probability theory and statistical physics has been the investigation of how the large-scale geometric properties of a space are reflected in the behaviour of random walks on that space. Regular infinite graphs arise naturally as Cayley graphs of finitely generated groups. An important example is the infinite $d$-ary tree $\bbT_d$ (also known as the Bethe lattice in statistical physics, where it has been widely used to simplify problems that are difficult to solve on regular lattices), which arises as the Cayley graph, with respect to the standard generators, of $\bbZ^{* d_1} * \bbZ_2^{*d_2}$, with $d = 2d_1 + d_2$, where $*$ denotes the free product of groups. The simple (i.e. nearest neighbour) random walk (SRW) on $\mathbb{T}_d$ is known to be ballistic for $d \ge 3$, with speed $\frac{d-2}{d}$, and, a fortiori, transient (see, e.g., Chapter 13 of \cite{lyons2017probability}). In general, the study of random walks on groups has uncovered profound connections between probability theory and geometric group theory.

Motivated by this, in this article, we introduce elephant random walks on finitely generated groups, extending the notion of random walks with long-range memory beyond classical lattices, the goal being to explore how such memory interacts with the geometry of the group. As a first step in this broader program, we focus on the infinite $d$-ary tree $\bbT_d$ for $d \geq 3$, viewed as the Cayley graph of a finitely generated group. We show that memory has no impact on the asymptotic speed (and hence recurrence/transience) of the walk, thus matching the escape rate $\frac{d - 2}{d}$ of the SRW. We also establish upper bounds on the rate of convergence to the asymptotic speed. These upper bounds exhibit a phase transition at $p = \frac{d + 1}{2d}$, below which they are of order $n^{-1/2}$ and slower otherwise. Numerical experiments suggest that these upper bounds are tight. We also obtain estimates on the return probability, which decays exponentially for any $0 \le p < 1/2$ such that $(p, d) \ne (0, 3)$. Based on extensive numerical experiments, we present several open problems on fluctuations around the limiting speed and the decay of the return probability in all memory regimes.

We would like to emphasize here that the non-commutative and non-Markovian nature of the ERW on a non-abelian group makes its analysis challenging. Indeed, due to the lack of the Markov property, a direct potential-theoretic approach is infeasible. For the same reason, it is not clear how tools from sub-additive ergodic theory could be applied. While our proofs are martingale-based, due to non-commutativity, they are of a rather different flavour from the (also martingale-based) techniques used to analyse elephant random walks on $\bbZ^d$. Indeed, for the MERW, the position of the walker and hence their distance from the origin can be written succinctly in terms of the step counts along the various axis directions in the past (see, e.g., Eq. (13) of \cite{qin2025recurrence}). Such a simple representation is, however, not possible on non-abelian groups, where the geometry of the Cayley graph enters non-trivially, interacting with step counts in a more intricate manner. It is necessary to decouple the urn process governing the step counts from the geometry of the Cayley graph in order to make progress.

The rest of the paper is organised as follows. In Section~\ref{sec:model}, we describe the model and state our main results. In Section~\ref{sec:prelim}, we set up some preliminaries and prove a crucial concentration estimate. Finally, Section~\ref{sec:proofs} collects the proofs of our main results. Section~\ref{sec:conc} contains some concluding remarks.

\section{The model and our main results}\label{sec:model}
\subsection{Random walk on finitely generated groups}
We first briefly recall the definition of random walks on (discrete) groups.

Let $\mu$ be a probability measure on a discrete group $\Gamma$. The (left-) $\mu$-random walk on $\Gamma$ is the (Markovian) stochastic process $(g_n \cdots g_1)_{n \ge 1}$, where the ``increments'' $g_i$ are i.i.d. random elements distributed as $\mu$.

Let $\Gamma$ be a finitely generated group, with a symmetric\footnote{$\cS$ being symmetric entails that $s \in \cS$ if and only if $s^{-1} \in \cS$.} generating set $\cS$. Recall that the (left-) Cayley graph $\Cay(\Gamma; \cS)$ of $\Gamma$ with respect to $\cS$ is a $|\cS|$-regular graph with vertex set $\Gamma$ and an edge between vertices $x, y \in \Gamma$ if and only if $yx^{-1} \in \cS$. If $\mu$ is the uniform distribution on $\cS$, then the $\mu$-random walk on $\Gamma$ is the simple random walk on $\Cay(\Gamma; \cS)$.

For more background on random walks on groups, we refer the reader to the books \cite{woess2000random, lyons2017probability, lalley2023random, yadin2024harmonic}.

\subsection{The ERW on finitely generated groups}
We now introduce a generalisation of the ERW to any finitely generated group $\Gamma$, with a symmetric generating set $\cS$ with $|\cS| = d$. We shall think of the elements of $\cS$ as the \emph{steps} of the walk. The elephant random walk on $\Cay(\Gamma; \cS)$ is a non-Markovian stochastic process $(w_n)_{n \ge 0}$ defined as follows: We start the walk at the identity, i.e. $w_0 = e$. We take the first step $g_1$ to be a uniformly random step from $\cS$. For $n \ge 1$, given past steps $g_1, \ldots, g_n$, we first choose a uniform random number $D \sim \Unif([n])$, where $[n] := \{1, \ldots, n\}$, and then set 
\begin{equation}\label{eq:elephant-step}
    g_{n + 1} = \begin{cases}
        g_D & \text{with probability $p$}, \\
        a \in \cS \setminus \{g_D\} & \text{with probability $\frac{1 - p}{d - 1}$ each}.
    \end{cases}
\end{equation}
Let $w_n$ denote the position of the walk at time $n$, which is encoded by the \emph{reduction} of the element $g_n \cdots g_1$ as per the group laws. Let $\Delta_n = \rho(w_n, e)$ denote the distance from the root $e$ after time $n$, where $\rho$ denotes the word-metric, giving the length of the reduced word.

We note that for the special value $p = 1/d$, the ERW is the same as the SRW on $\Cay(\Gamma; \cS)$.

\begin{remark}
It is clear that the above model generalises the ERW on $\bbZ^d$. Indeed the $d$-dimensional lattice is the Cayley graph of the abelian group $\bbZ^d$ with respect to the generating set $\cS = \{\pm \be_1, \ldots, \pm \be_d\}$, where the $\be_i$'s are the canonical basis vectors of $\bbZ^d$.
\end{remark}

\subsection*{An alternative description as a step-reinforced random walk} 
The ERW is a special case of the so-called step-reinforced random walks (SRRW). \cite{kim2014anomalous} studied an SSRW on $\bbZ$ which is equivalent to the ERW (this was pointed out by \cite{kursten2015comment};  see also \cite{kursten2016random}). Since then SSRWs with more general step distributions have been studied by various authors, including \cite{businger2018shark, bertoin2020universality, bertoin2020noise, bertoin2021scaling, bertenghi2021asymptotic, bertenghi2022joint, hu2024strong, qin2025recurrence_step}. One may readily generalise the positively step reinforced random walk (pSSRW) to groups as follows: Given a probability measure $\mu$ on a group $\Gamma$, we let $g_1 \sim \mu$.  Given steps $g_1, 
\ldots, g_n$, at time $n + 1$, we
\begin{itemize}
    \item sample $D \sim \Unif([n])$;
    \item with probability $\tilde{p}$, set $g_{n + 1} = g_D$;
    \item with probability $1 - \tilde{p}$, set $g_{n + 1}$ to be a independent random element distributed as $\mu$.
\end{itemize}
The process $(g_n \cdots g_1)_{n \ge 1}$ obtained in this way is the pSSRW on $\Gamma$ with reinforcement parameter $\tilde{p}$ and step distribution $\mu$.

Given a finitely generated group $\Gamma$ with generating set $\cS$, let $\mu$ be the uniform distribution on $\cS$. The resulting pSSRW is a nearest-neighbour and symmetric random walk on $\Cay(\Gamma; \cS)$ with reinforcement parameter $\tilde{p}$. Notice that with $\tilde{p} \in [0, 1)$, this walk is equivalent to the ERW on $\Cay(\Gamma; \cS)$ with memory parameter $p = \tilde{p} + \frac{1 - \tilde{p}}{d} \in [\frac{1}{d}, 1)$. This gives an alternative construction of the ERW on $\Cay(\Gamma; \cS)$ for $p \in [1/d, 1)$.

To cover the complementary regime, we shall generalise the \emph{counterbalanced random walks} of Bertoin \cite{bertoin2023counterbalancing}, also called negatively step-reinforced random walks (nSSRW). For generalising Bertoin's counterbalancing mechanism to a group setting, we need the concept of the ``opposite'' of a recalled step. Taking group inverses is not an option here for the simple reason that an element may be its own inverse. We shall therefore confine ourselves to finitely generated groups. Also, to keep things simple, we shall assume that the step distribution is uniform over the generating set. With these in mind, let $\Gamma$ be a finitely generated group with generating set $\cS$. A (symmetric) negatively step-reinforced random walk (nSSRW) on $\Cay(\Gamma; \cS)$, with reinforcement parameter $\tilde{p}$, adds a counterbalancing mechanism by \emph{not walking} in the direction $g_D$ with probability $\tilde{p}$. Thus beginning with an uniform random step $g_1$, and given $g_1, \ldots, g_n$, at time $n + 1$, we
\begin{itemize}
    \item sample $D \sim \Unif([n])$;
    \item with probability $\tilde{p}$, set $g_{n + 1}$ to be a uniformly random step from $\cS \setminus \{g_D\}$;
    \item with probability $1 - \tilde{p}$, set $g_{n + 1}$ to be a uniformly random step.
\end{itemize}
The process $(g_n \cdots g_1)_{n \ge 1}$ obtained in this way is the (symmetric) nSSRW on $\Cay(\Gamma; \cS)$ with reinforcement parameter $\tilde{p}$. It is not difficult to check that with $\tilde{p} \in (0, 1]$, this process is equivalent to the ERW on $\Cay(\Gamma; \cS)$ with memory parameter $p = \frac{1 - \tilde{p}}{d} \in [0, \frac{1}{d})$.

\subsection{ERW on Cayley trees}
\begin{figure}[!t]
    \centering
    \begin{tabular}{cc}
        \includegraphics[width=0.5\textwidth]{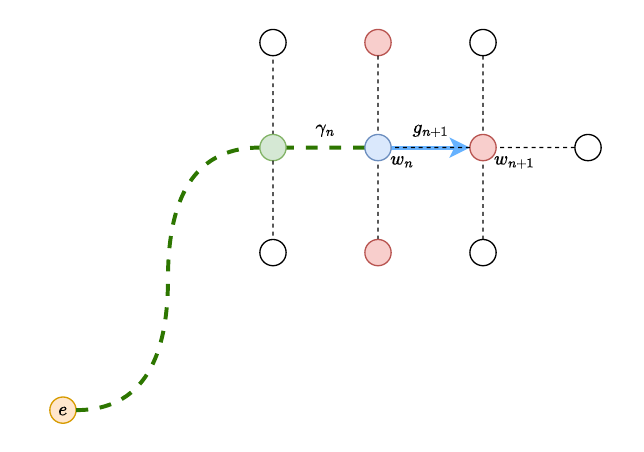} & 
        \includegraphics[width=0.5\textwidth]{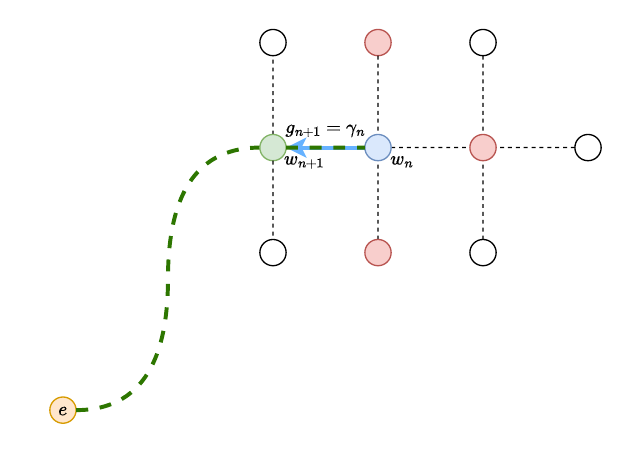} \\
        (a) & (b) 
    \end{tabular}
    \caption{ERW on $\bbT_4$: At $w_{n} \ne e$, one either (a) moves away from the root, i.e. $g_{n + 1} \ne \gamma_n$, so that $\Delta_{n + 1} - \Delta_{n} = 1$, or (b) moves towards it, i.e. $g_{n + 1} = \gamma_n$, so that $\Delta_{n + 1} - \Delta_{n} = -1$. The green dotted line depicts the unique geodesic connecting $e$ and $w_n$.}
    \label{fig:schematic}
\end{figure}

As stated in the introduction, the goal of this paper is to study ERW on finitely generated \emph{non-abelian} infinite groups, by studying perhaps the simplest case where the Cayley graph is the Bethe lattice, i.e. the infinite $d$-regular tree $\bbT_d$, with $d \ge 3$. Writing $\bbZ \cong \langle a, a^{-1} \rangle$ and $\bbZ_2 \cong \langle b \mid b^2 = e\rangle$, one may realise $\bbT_d$ as a Cayley graph by considering the free product $\bbZ^{*d_1} * \bbZ_2^{*d_2}$, with $d = 2d_1 + d_2$. Indeed, $\Gamma = \bbZ^{*d_1} * \bbZ_2^{*d_2}$ is generated by $\cS = \{a_1, a_1^{-1}, \ldots, a_{d_1}, a_{d_1}^{-1}\} \cup \{b_1, \ldots, b_{d_2}\}$, where the only non-trivial relations satisfied by these generators are $b_j^2 = e$ for $j = 1, \ldots, d_2$.

Fix a finitely generated group $\Gamma$, with generating set $\cS$, whose Cayley graph is $\bbT_d$. The identity element $e$ of the group $\Gamma$ is to be thought of as the root of our tree. We remark here that the walk is dependent on the group and its generators (since we use the generators as steps). For instance, on $\bbT_4$, the walks encoded by the group $\bbZ * \bbZ$ will be different from the walk encoded by $\bbZ_2^{* 4}$. In the former, if $a$ is a generator, the sample path $(a^n)_{n \ge 0}$ will move away from the root, whereas in the latter it will alternate between $a$ and $e$.

This is a good place to introduce some objects that will be key in our subsequent analysis. Suppose $w_n \ne e$. As we are walking on a tree, there is a unique geodesic joining $e$ and $w_n$. We denote by $\gamma_n \in \cS$ the unique next step that would decrease the distance by $1$. Thus if $g_{n + 1} = \gamma_n$, one would have $\Delta_{n + 1} = \Delta_n - 1$, and otherwise, the distance would go up: $\Delta_{n + 1} = \Delta_n + 1$. See Figure~\ref{fig:schematic} for an illustration.

For $a \in \cS$, let $N_n(a) := \#\{1\le k \le n : g_k = a\}$ denote the number of $a$-steps taken by time $n$. Define
\begin{equation}\label{eq:ell_n}
    \ell_n := \begin{cases}
        N_n(\gamma_n) & \text{if } w_n \ne e, \\
                    0 & \text{otherwise.}
    \end{cases}
\end{equation}
We shall use the convention that $\frac{\ell_0}{0} = 0$. We also define 
\begin{equation}\label{eq:Xi}
  \Xi_n := \frac{1}{n}\sum_{k = 0}^{n - 1} \bigg(\frac{\ell_k}{k} - \frac{\ind(\Delta_k > 0)}{d}\bigg),
\end{equation}
a quantity whose concentration properties will be fundamental to our analysis.

\subsection{Asymptotic notation.} We shall use the notations $A_n \lesssim B_n$ or $B_n \gtrsim A_n$ to mean that there is some constant $C > 0$, not depending on $n$, such that $A_n \le C B_n$ for all $n \ge n_0$, for some $n_0 \ge 1$.  The constant $C$ will typically depend on various parameters such as $d, p, m$, etc., which will be clear from the context. The notation $A_n \asymp B_n$ shall mean that both $A_n \lesssim B_n$ and $B_n \lesssim A_n$ hold true. For random variables, we use the standard $O_P$ and $o_P$ notations.

\subsection{Main results}
Our first main result shows that the asymptotic speed/escape rate of the ERW on a Cayley tree does not depend on the memory parameter $p$.
\begin{theorem}[Escape rate]\label{thm:LLN}
    We have for any $p \in [0, 1)$ that
    \[
        \lim_{n \to \infty} \frac{\Delta_n}{n} = \frac{d - 2}{d} \quad \text{a.s.}
    \]
\end{theorem}

\begin{remark}\label{rem:p=1}
    If $p = 1$, then the walker always repeats the first step, and hence $\frac{\Delta_n}{n} = 1$ when $\Gamma$ is the free group $\bbZ^{*d}$ with $2d$ generators. On the other hand, for $\Gamma = \bbZ_2^{*d}$, we have $\frac{\Delta}{n} \le \frac{1}{n}$. In general, a deterministic limit may not exist (e.g., on $\bbZ^{* d_1} * \bbZ_2^{* d_2}$, with $d_1, d_2 \ge 1$, where it is easy to see that $\frac{\Delta_n}{n} \convas \Ber(\frac{2d_1}{2d_1 + d_2})$; see Figure~\ref{fig:plots}-(b) for an empirical demonstration of this phenomenon on $\bbZ * \bbZ_2^{*2}$).
\end{remark}

Let $p_d = \frac{d + 1}{2d}$. We shall refer to the regimes $p < p_d$, $p = p_d$, and $p > p_d$ as the \emph{subcritical}, \emph{critical} and \emph{supercritical} regimes, respectively. We remark that $p_{2d}$ is also the critical probability of the MERW on $\bbZ^d$.

Let
\[
    r_n \equiv r_n(p, d) := \begin{cases}
        n^{1/2} & \text{if } p < p_d, \\
        (\frac{n}{\log n})^{1/2} & \text{if } p = p_d, \\
        n^{\frac{d(1 - p)}{d - 1}} & \text{if } p > p_d.
    \end{cases}
\]
Note that $\frac{d(1 - p)}{d - 1} < 1/2$ for $p > p_d$. We then have the following results on the rate of convergence to the asymptotic speed.
\begin{theorem}[Rate of convergence to the escape rate]\label{thm:rate}
    Let $p \in [0, 1)$. Then for any $m \in [1, 2]$,
    \begin{equation}\label{eq:upper-bd}
        \bbE \bigg|\frac{\Delta_n}{n} - \frac{d - 2}{d}\bigg|^m \lesssim r_n^{-m}.
    \end{equation}
    In the supercritical regime, \eqref{eq:upper-bd} holds for a larger range of $m$, namely, for all $m \in [1, \frac{d - 1}{d(1 - p)}]$.

    In fact, in the subcritical regime $p < p_d$, in a neighbourhood of $1/d$, we can show a matching lower bound: For any $m \in [1, 2)$, there is $\epsilon > 0$ such that for all $p \in (1/d - \epsilon, 1/d + \epsilon)$, one has
    \begin{equation}\label{eq:lower-bd}
        \bbE \bigg|\frac{\Delta_n}{n} - \frac{d - 2}{d}\bigg|^m \gtrsim n^{-m/2}.
    \end{equation}
\end{theorem}

\begin{figure}[!t]
\centering
\begin{tabular}{ccc}
    \includegraphics[width = 0.315\textwidth]{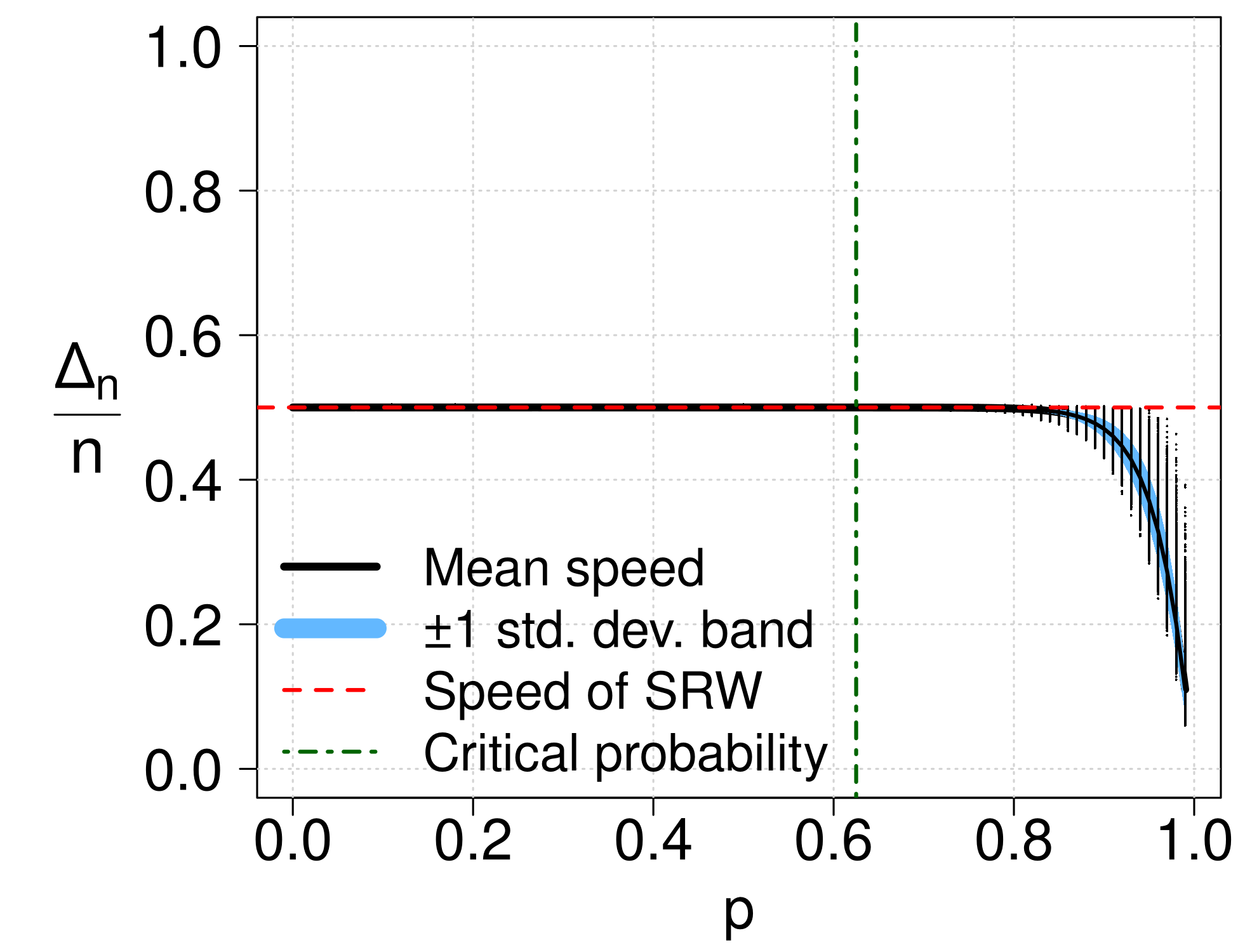} &
    \includegraphics[width = 0.315\textwidth]{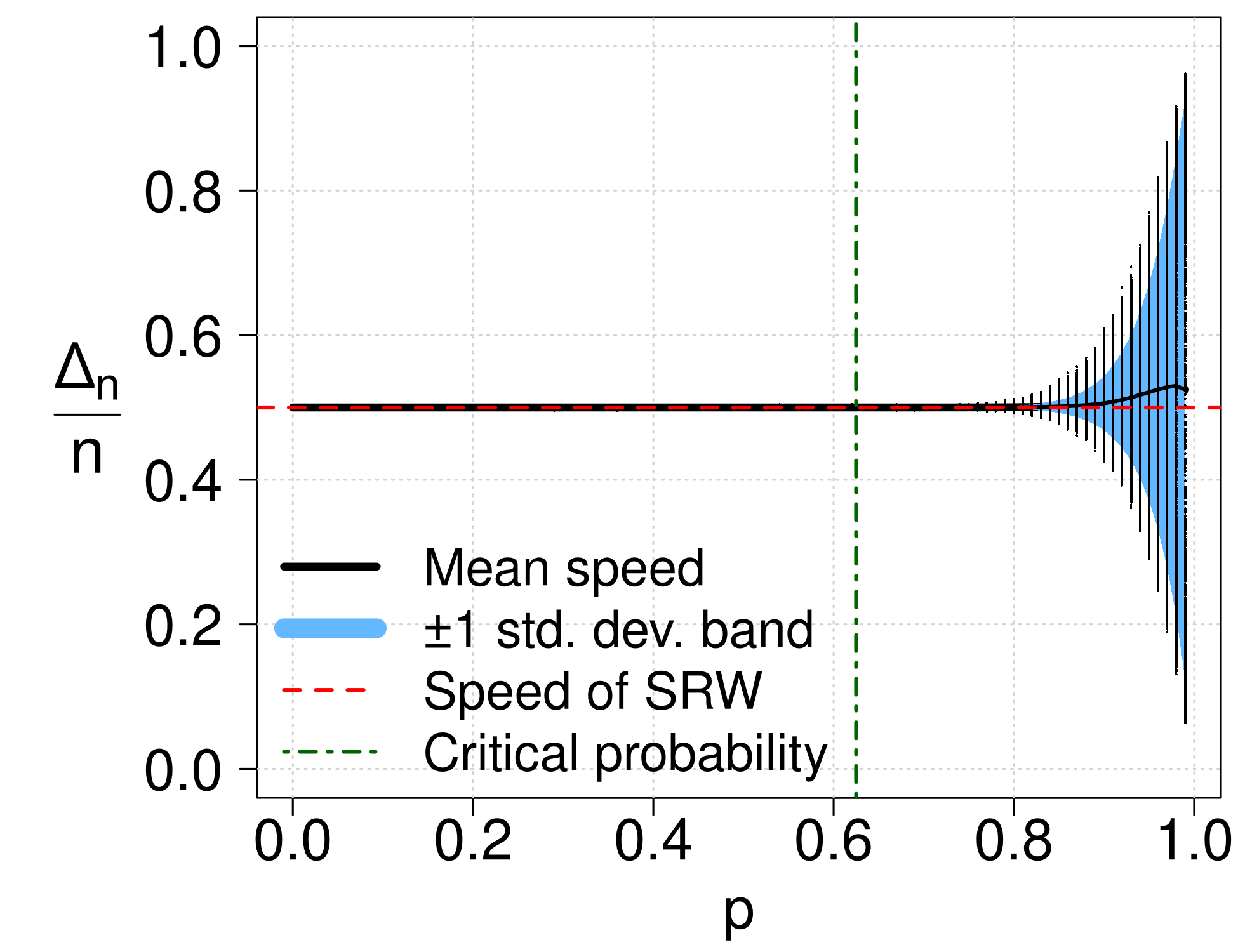} &
    \includegraphics[width = 0.315\textwidth]{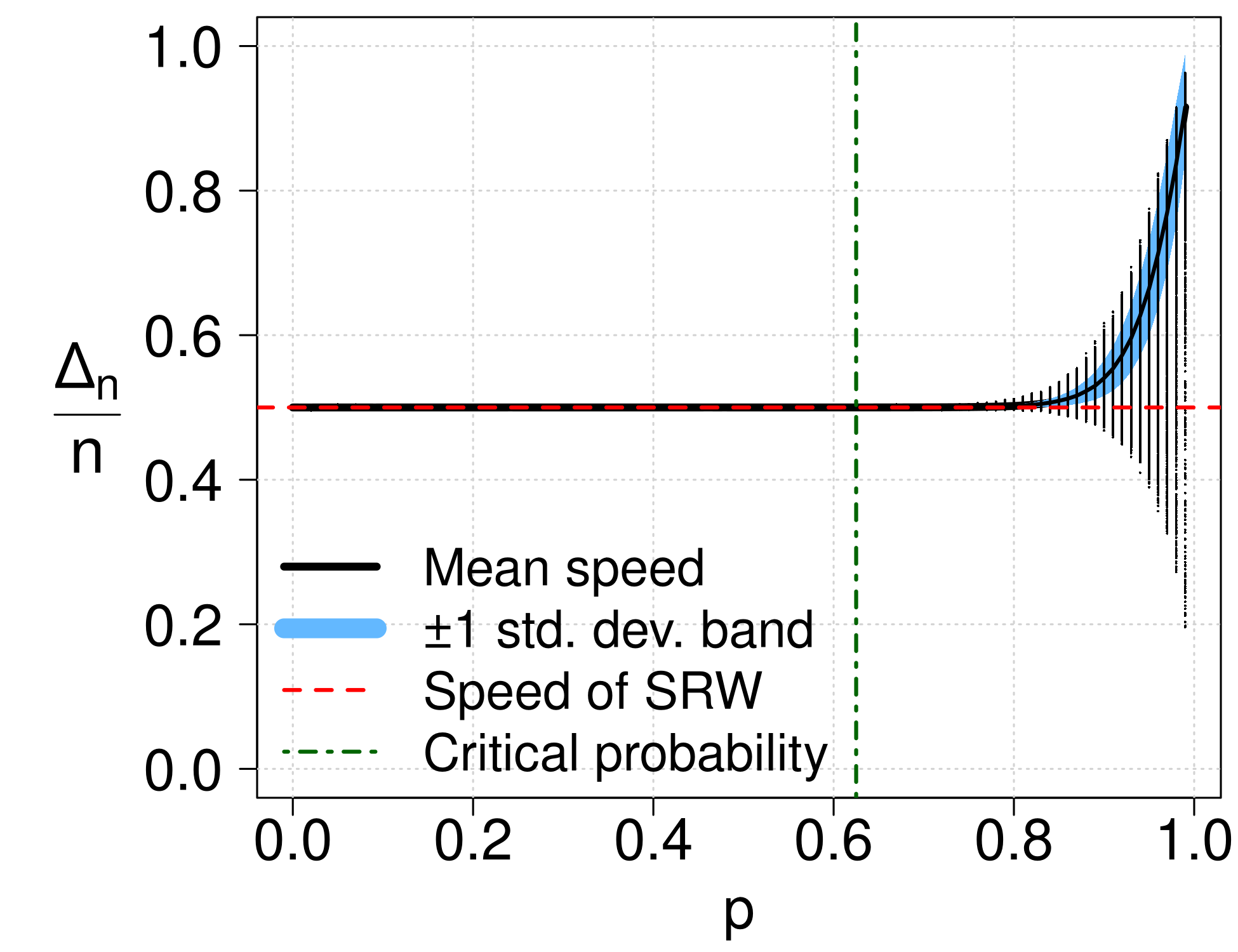} \\
    \hspace{2.25em}(a) & \hspace{2.25em}(b) & \hspace{2.25em}(c)
\end{tabular}
\caption{Escape rate of the ERW on the Cayley graphs of (a) $\bbZ_2^{*4}$, (b) $\bbZ * \bbZ_2^{*2}$, (c) $\bbZ^{*2}$, based on $10^4$ Monte Carlo runs with $n = 10^6$ steps each; the red horizontal line corresponds to $\frac{d - 2}{d} = \frac{1}{2}$, the escape rate of SRW. The observed deviation, near $p = 1$, of the mean speed from the theoretical value of $\frac{1}{2}$ reflects the increasingly slower rates of convergence.}
\label{fig:plots}
\end{figure}

\begin{corollary}
    $r_n \big(\frac{\Delta_n}{n} - \frac{d - 2}{d}\big)$ is a tight sequence.
\end{corollary}

The simulations presented in Figures~\ref{fig:plots} and \ref{fig:hist} empirically support the conclusions of Theorems~\ref{thm:LLN} and \ref{thm:rate}. While the upper bounds of Theorem~\ref{thm:rate} do not imply that the rate of convergence to the asymptotic speed is $r_n^{-1}$, our simulations suggest that this is the case.

\begin{conjecture}
Is it true that for all $m \ge 1$,
\[
    \bbE \bigg|\frac{\Delta_n}{n} - \frac{d - 2}{d}\bigg|^m \asymp r_n^{-m}?
\]
\end{conjecture}

\begin{figure}[!t]
\centering
\begin{tabular}{ccc}
    \includegraphics[width = 0.27\textwidth]{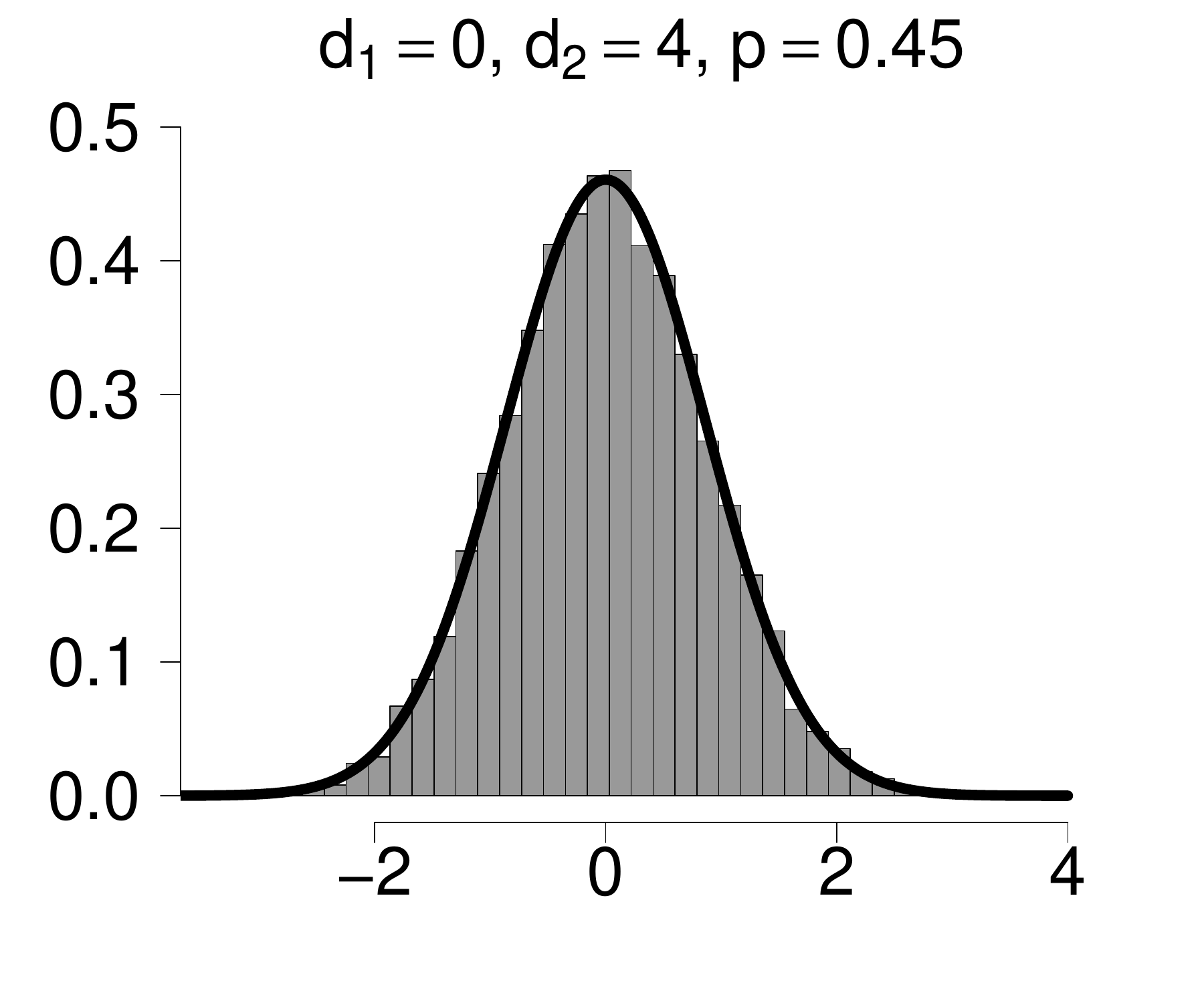} &
    \includegraphics[width = 0.27\textwidth]{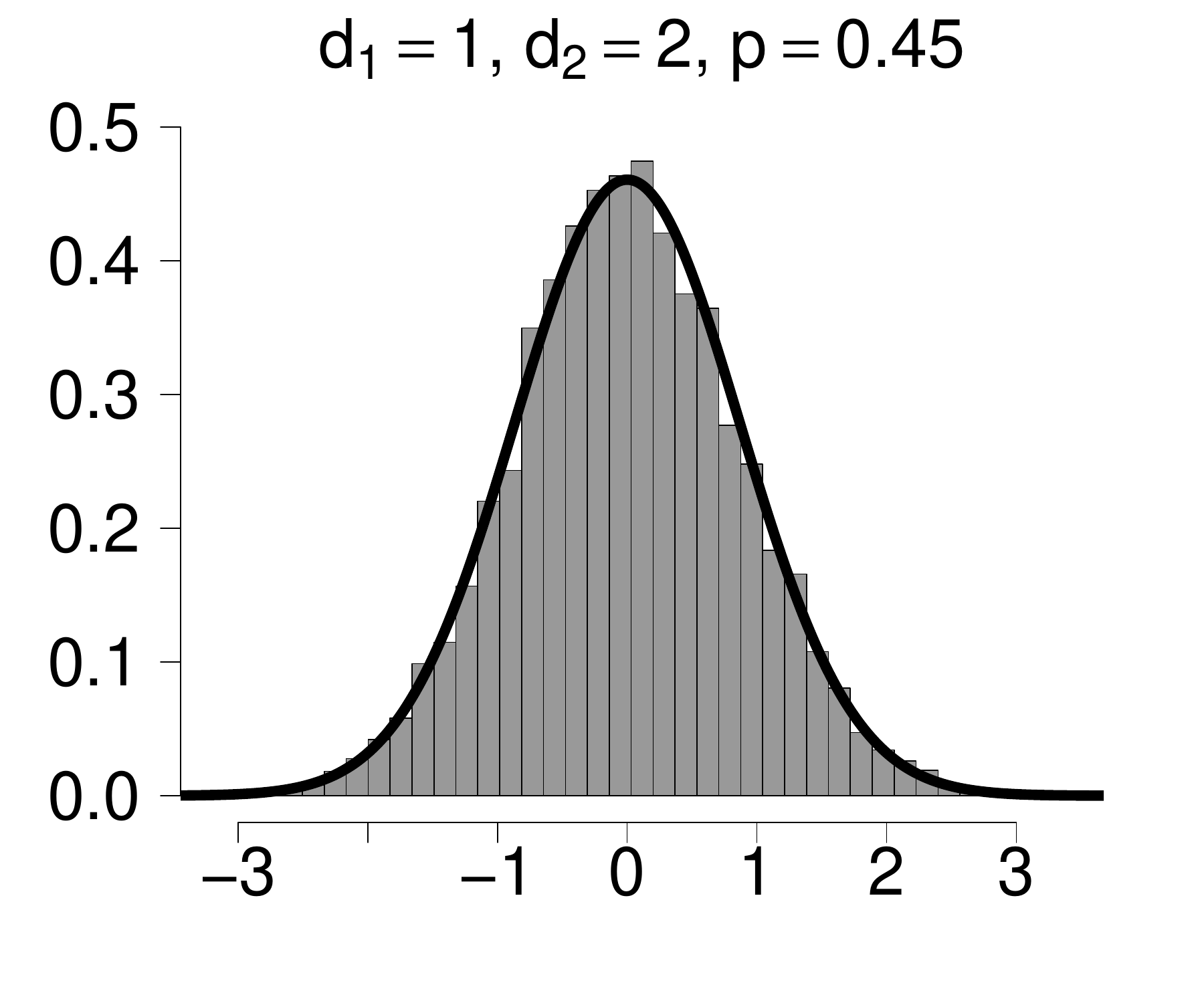} &
    \includegraphics[width = 0.27\textwidth]{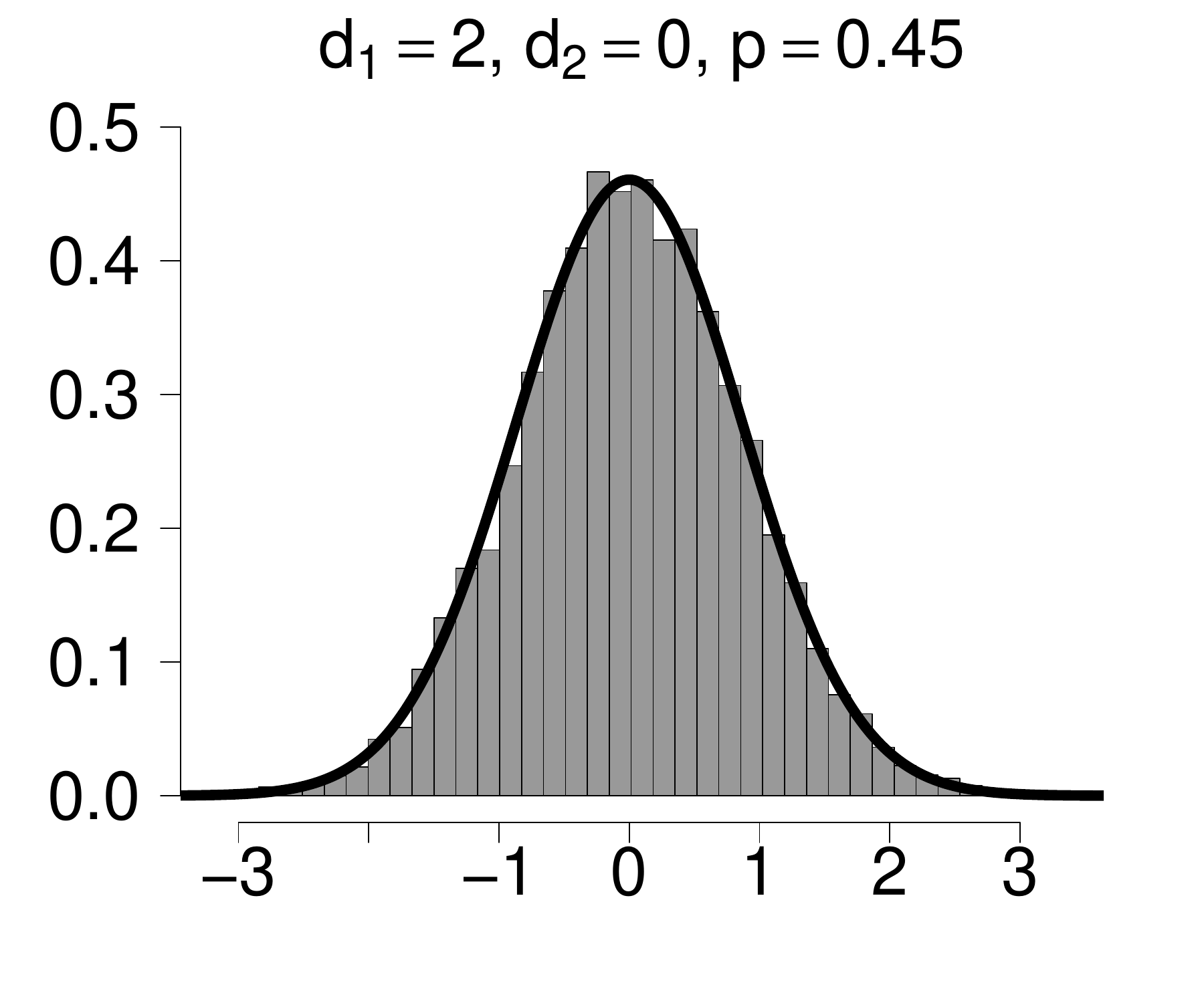} \\
    \includegraphics[width = 0.27\textwidth]{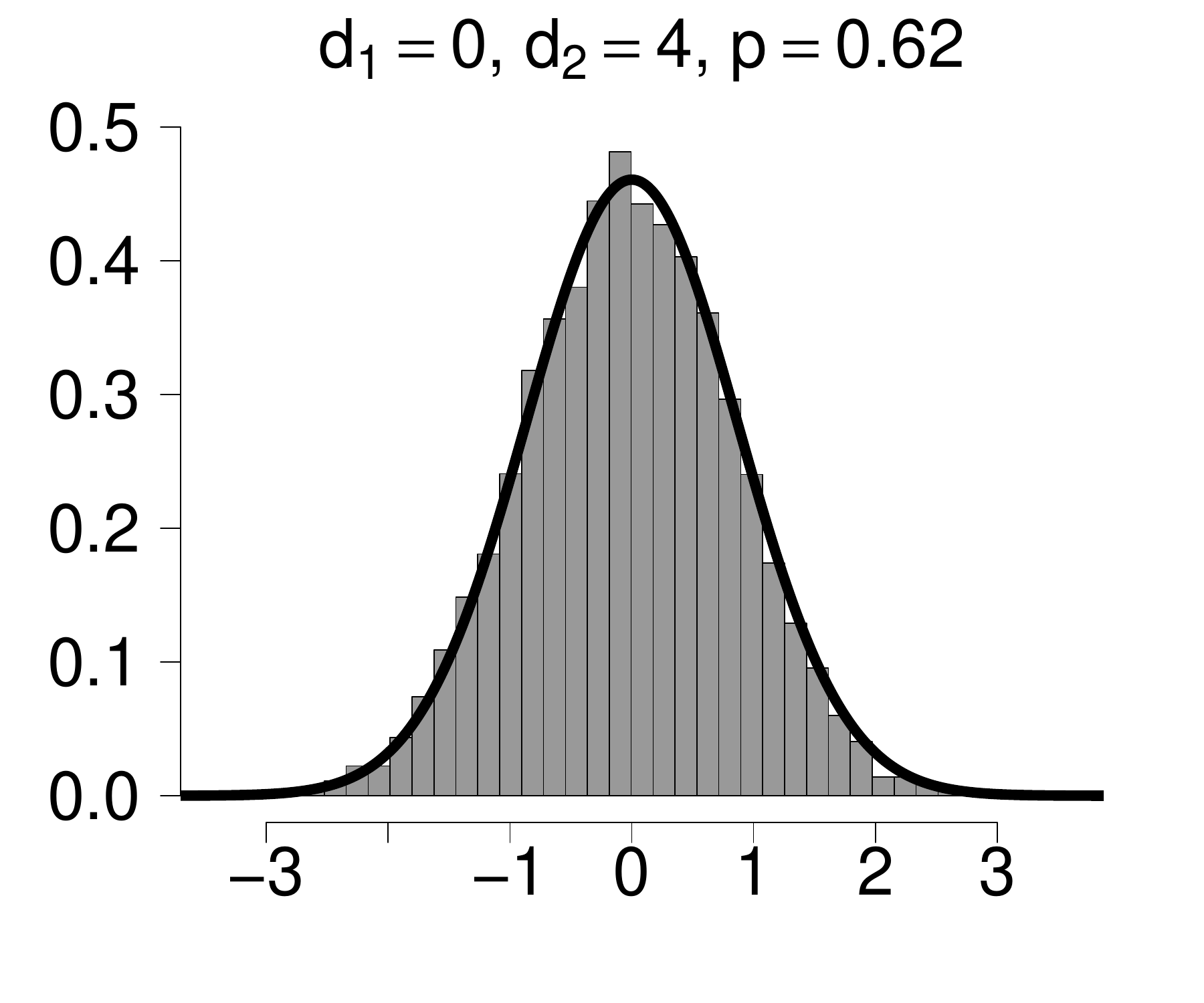} &
    \includegraphics[width = 0.27\textwidth]{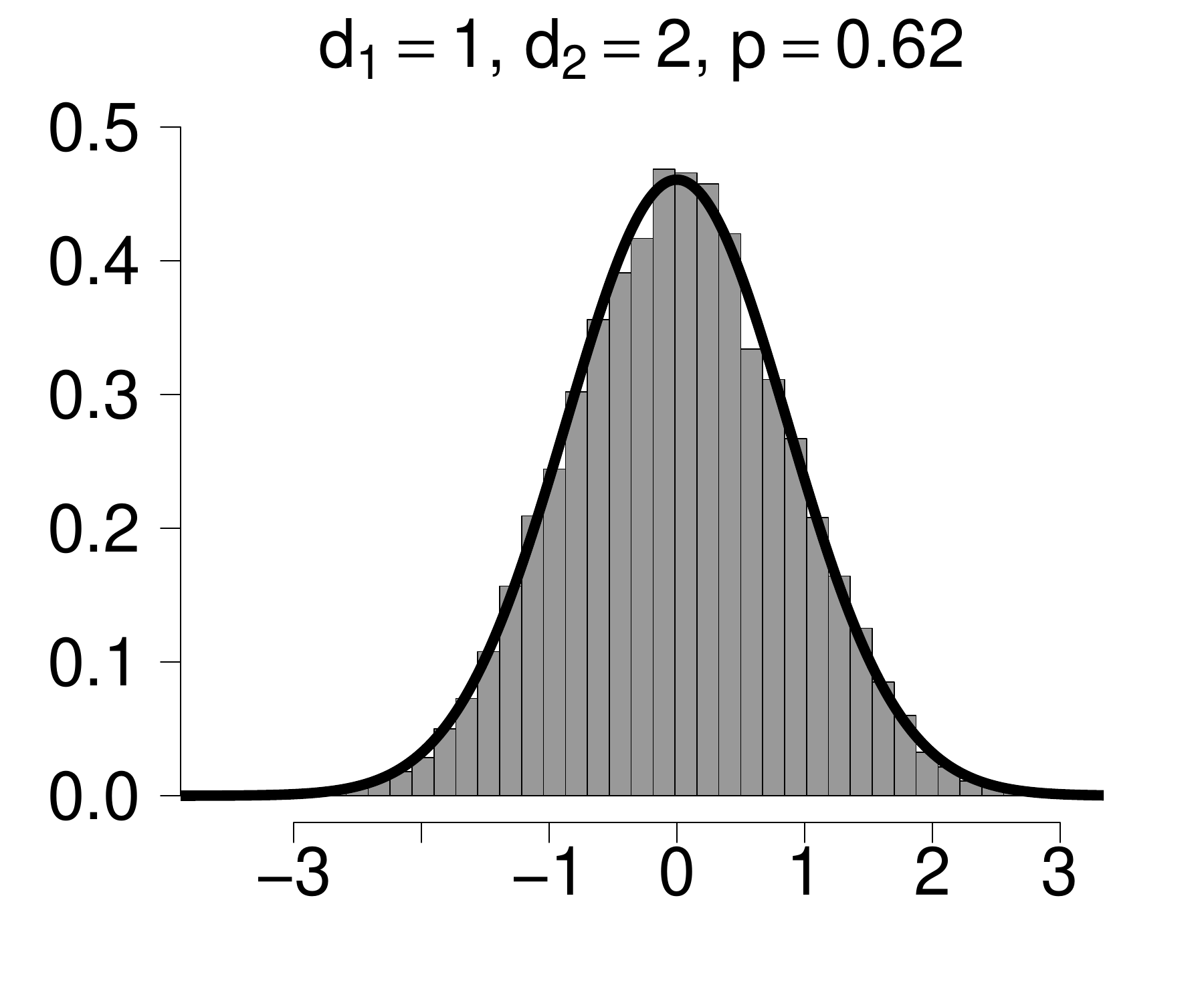} &
    \includegraphics[width = 0.27\textwidth]{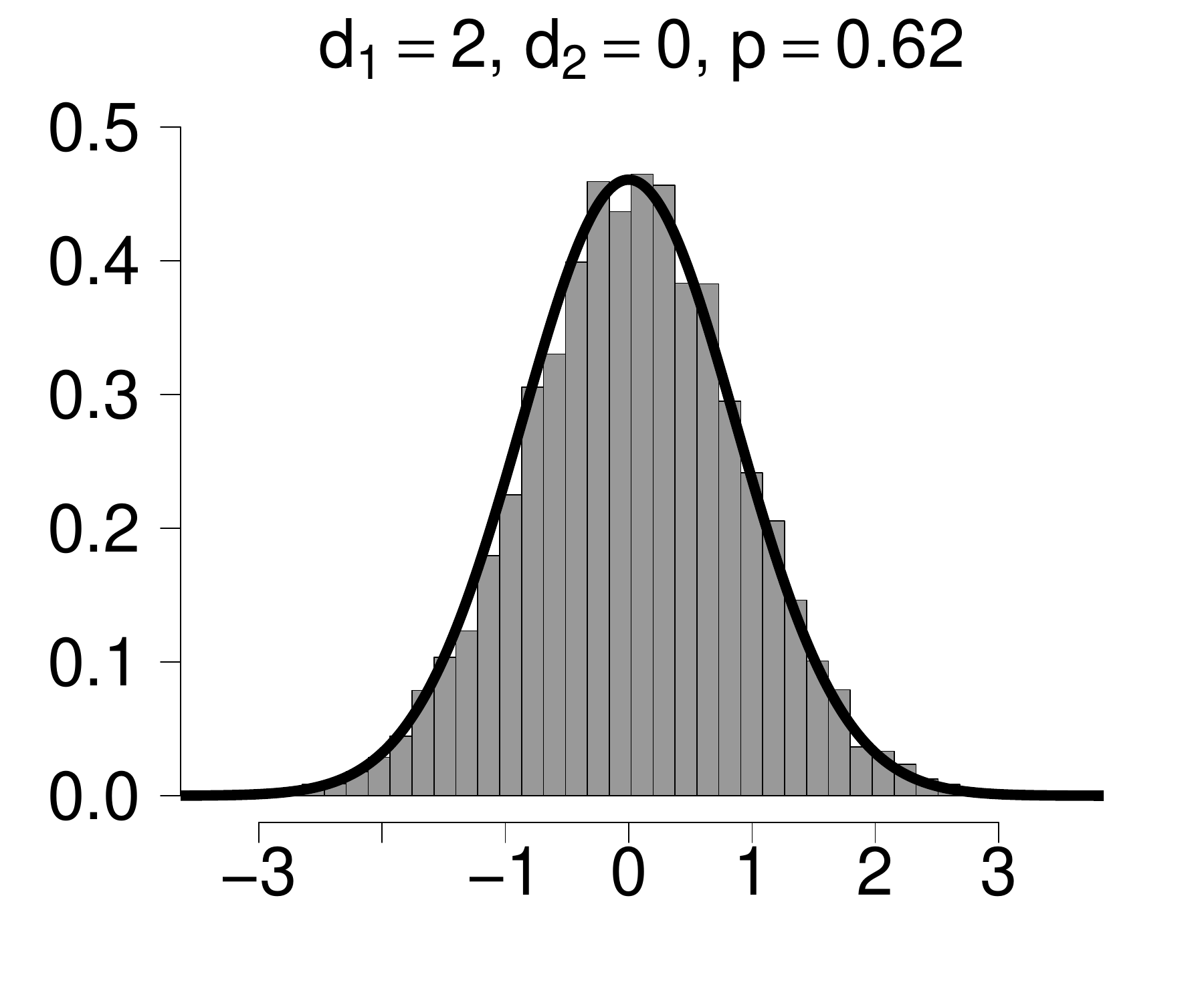} \\
    \includegraphics[width = 0.27\textwidth]{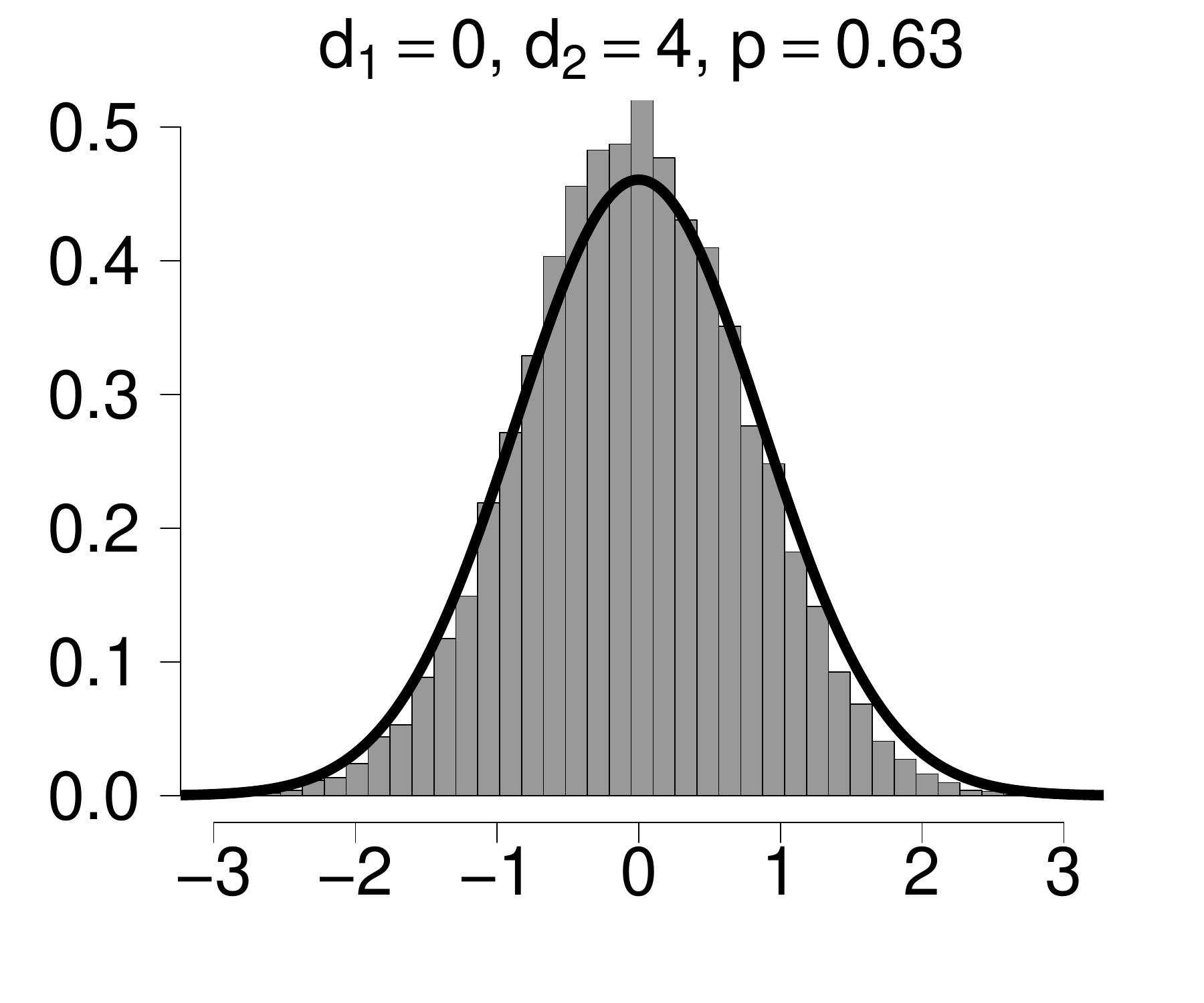} &
    \includegraphics[width = 0.27\textwidth]{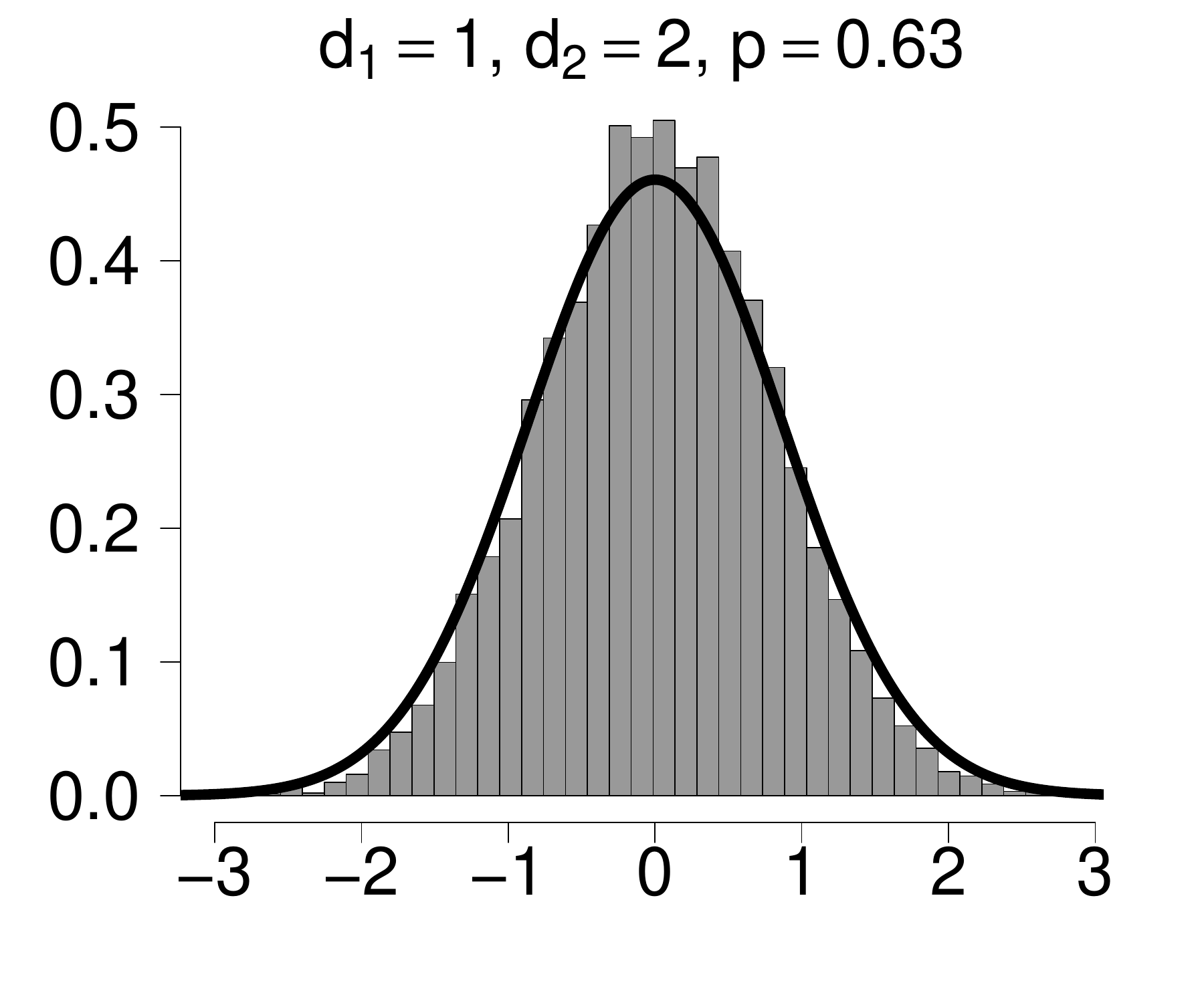} &
    \includegraphics[width = 0.27\textwidth]{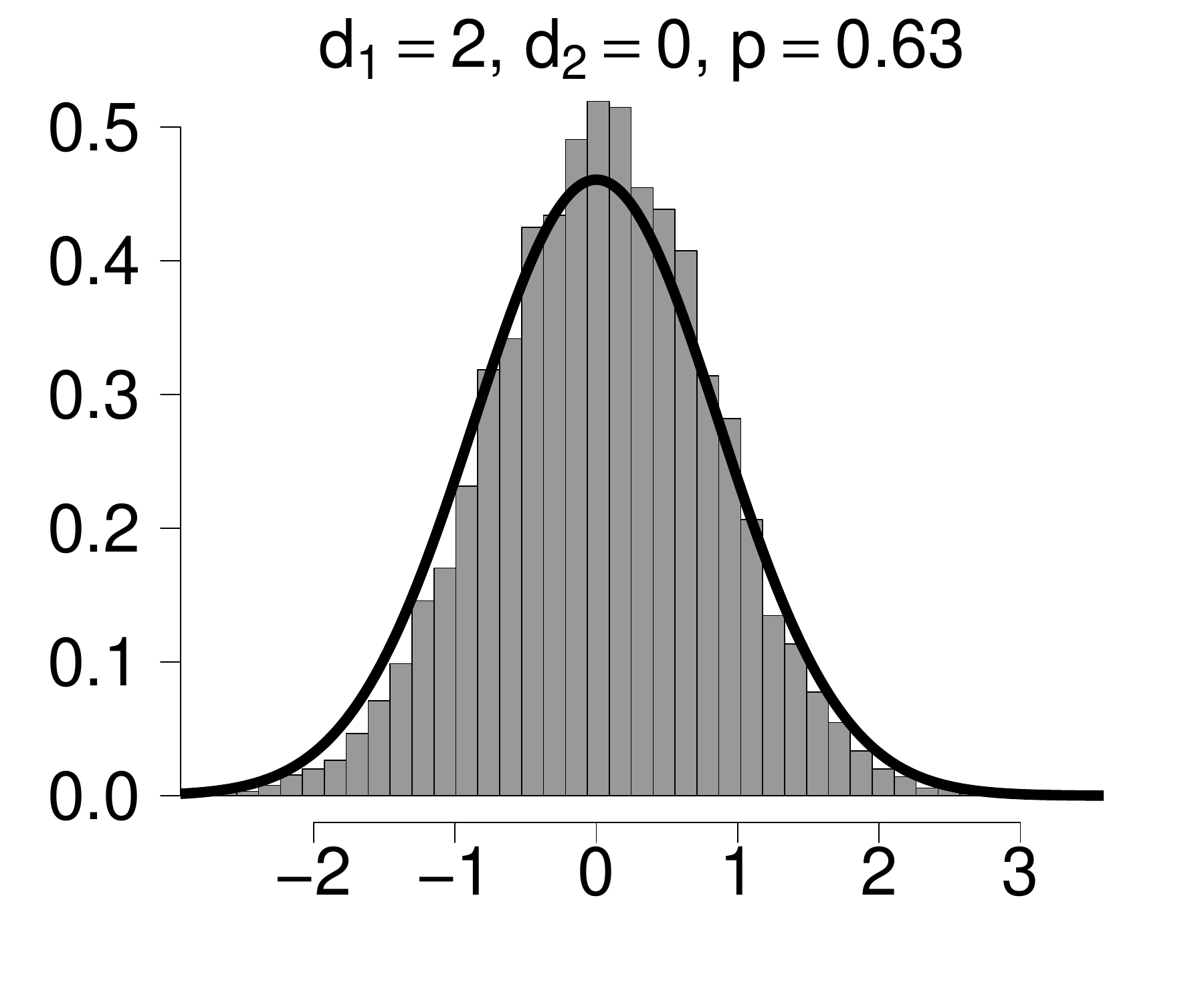} \\
    \includegraphics[width = 0.27\textwidth]{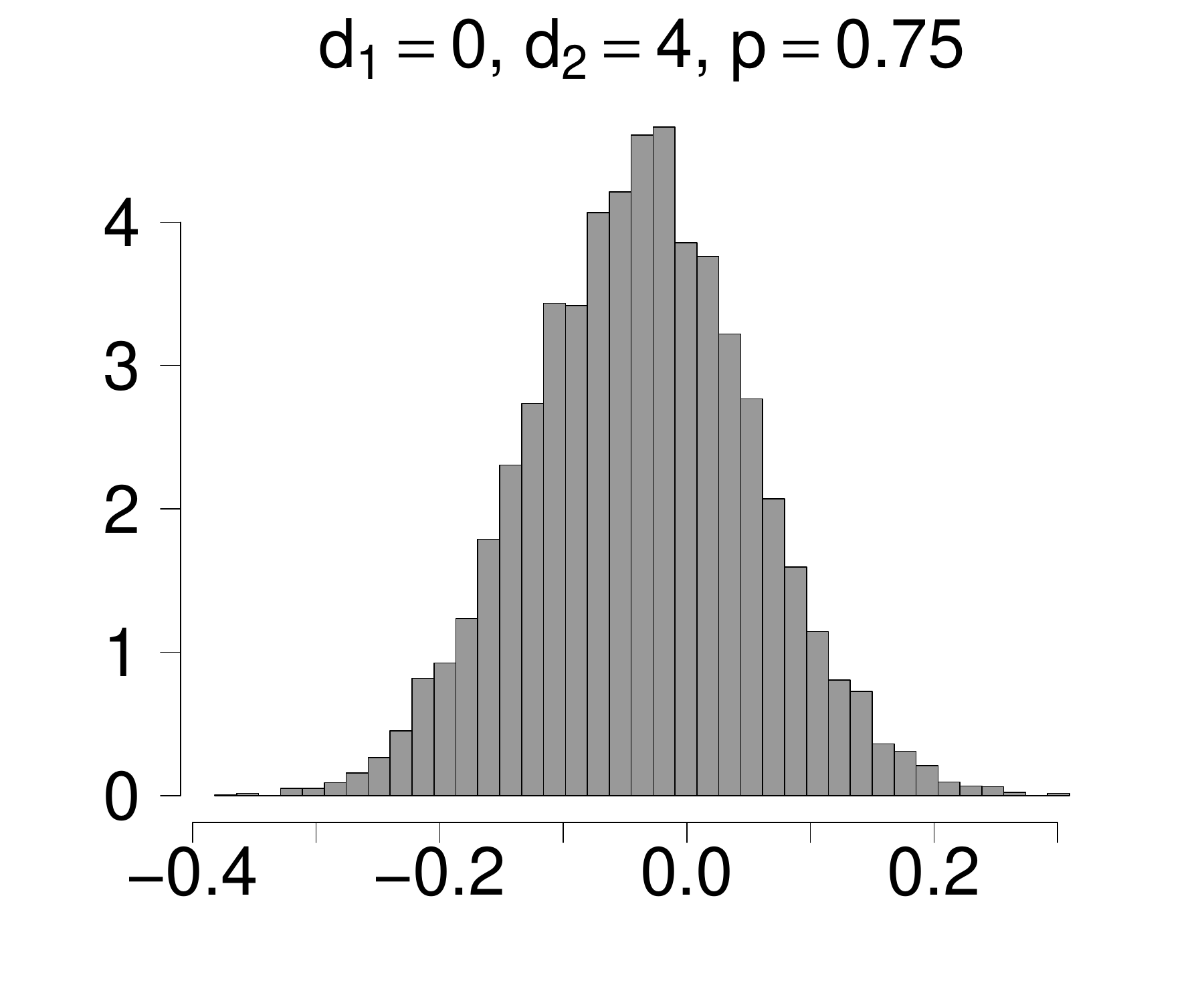} &
    \includegraphics[width = 0.27\textwidth]{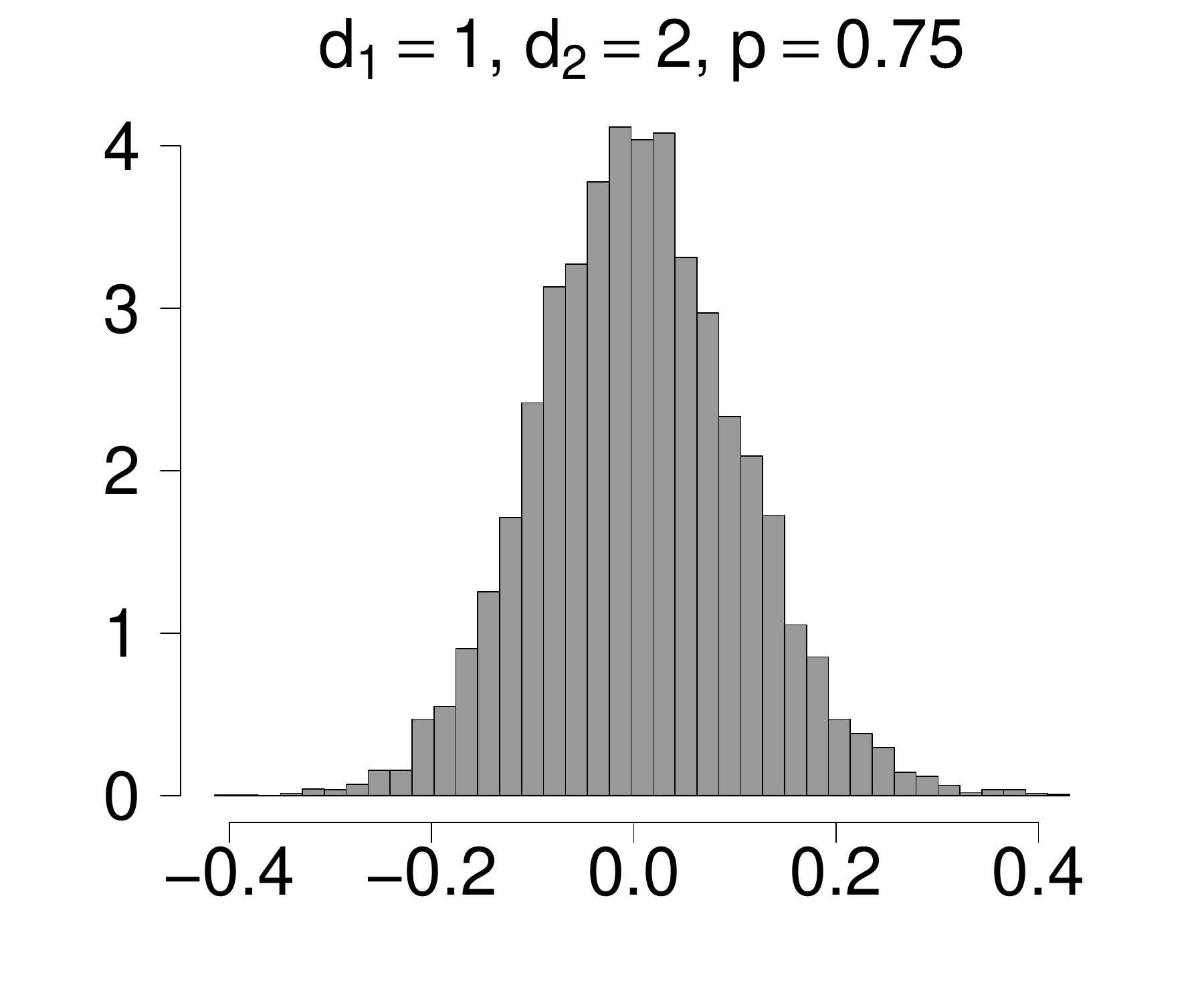} &
    \includegraphics[width = 0.27\textwidth]{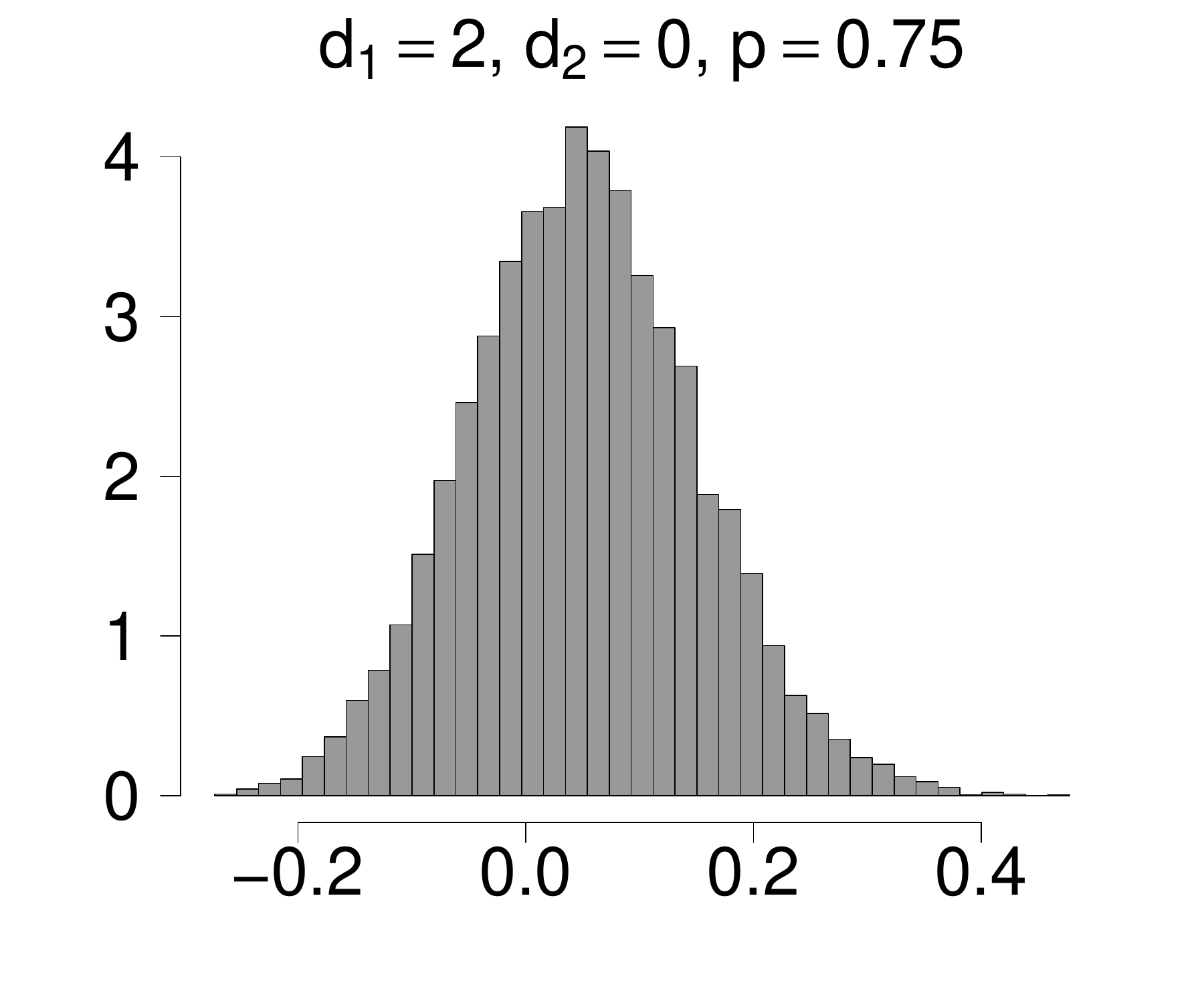} \\
    \includegraphics[width = 0.27\textwidth]{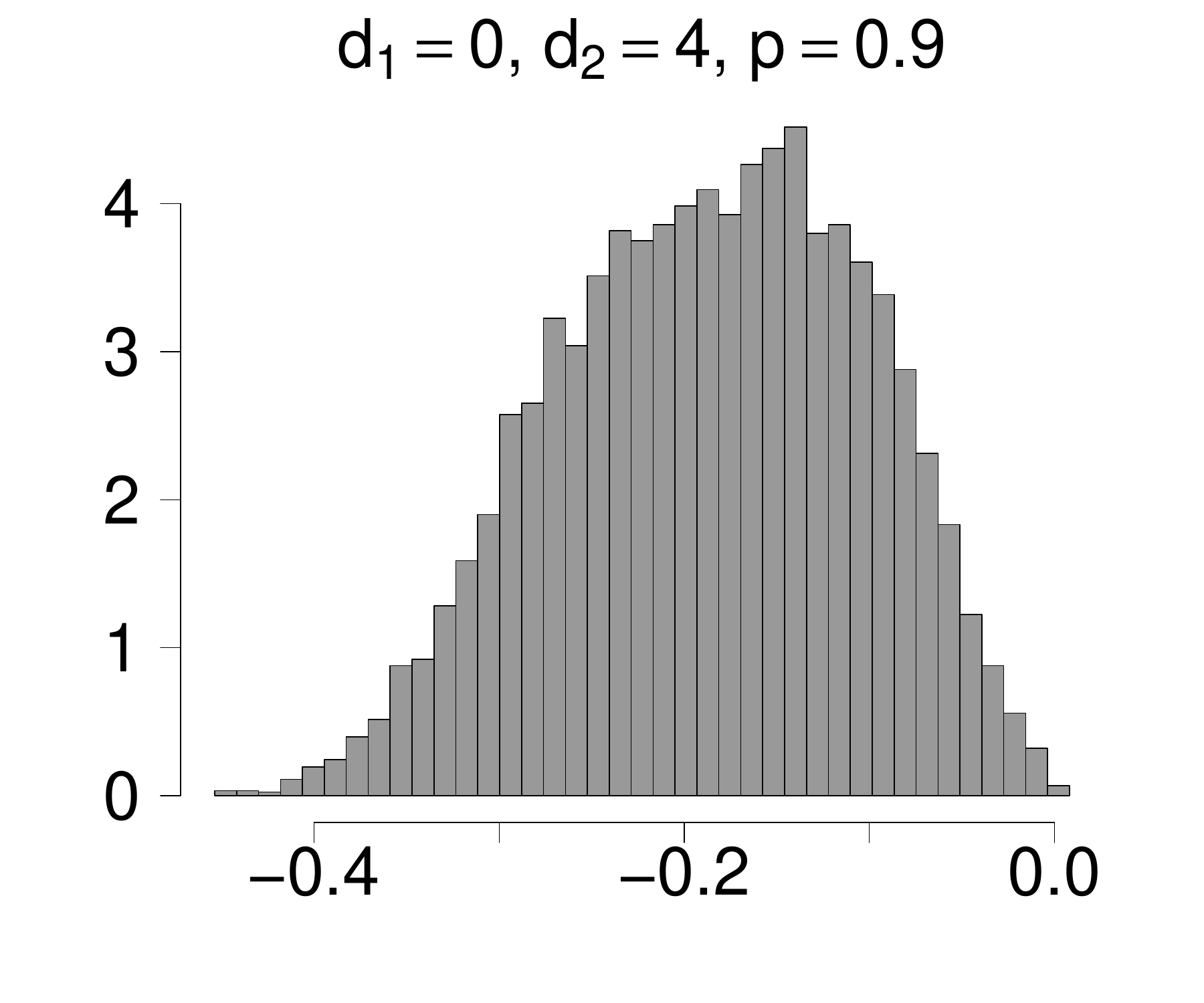} &
    \includegraphics[width = 0.27\textwidth]{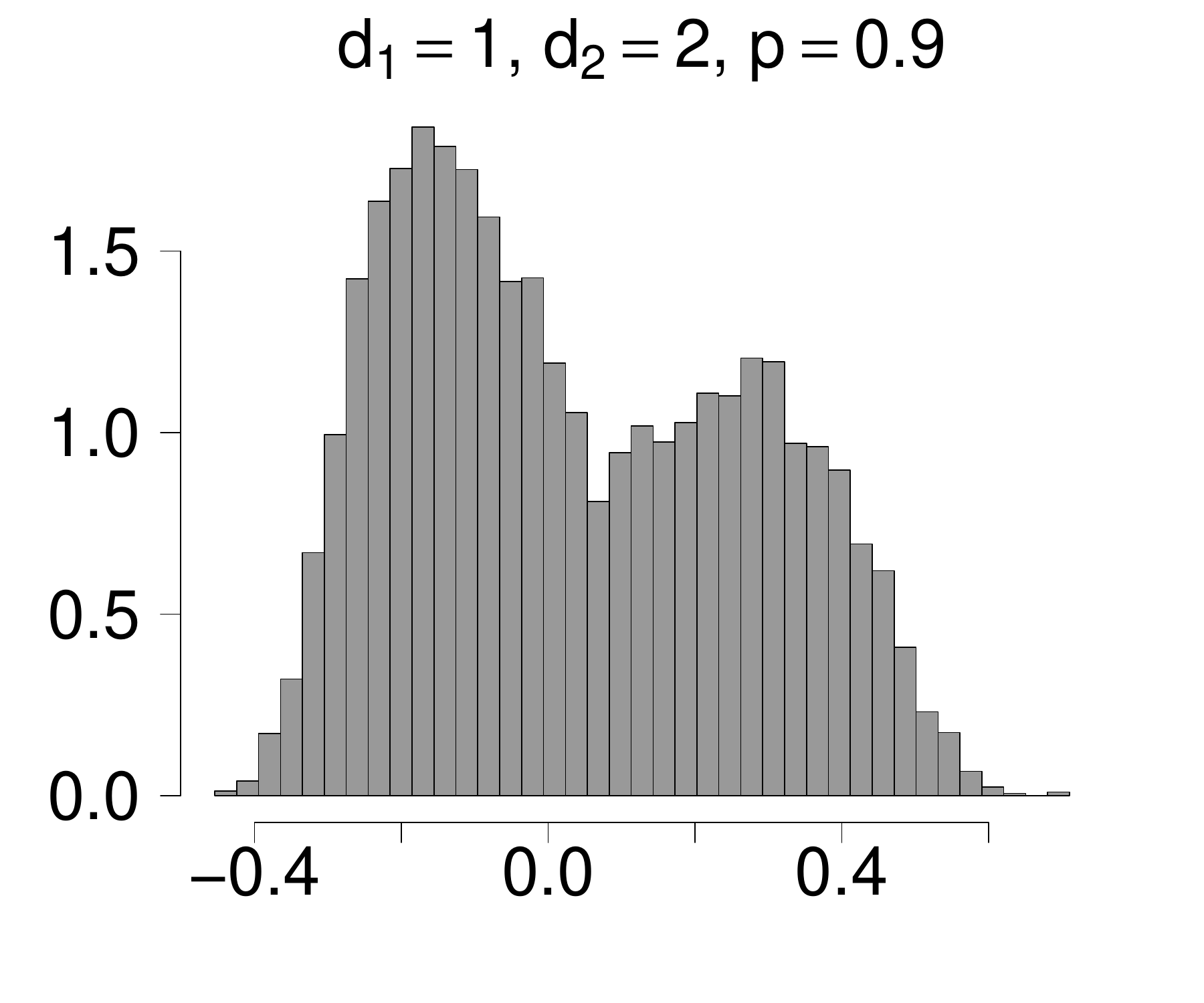} &
    \includegraphics[width = 0.27\textwidth]{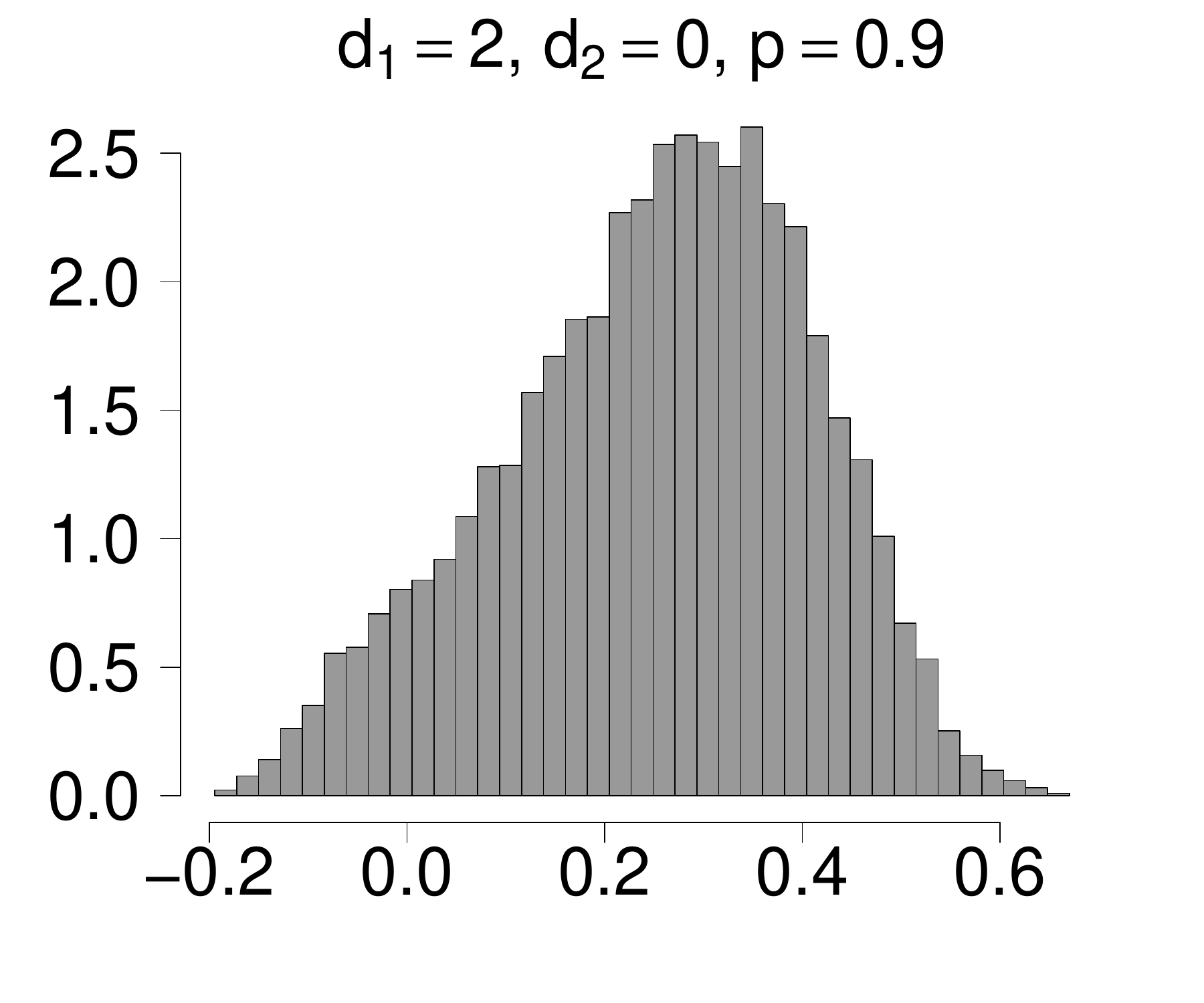} \\
   \hspace{1.5em}(a) & \hspace{1.5em}(b) & \hspace{1.5em}(c)
\end{tabular}
\caption{Histograms of $r_n\big(\frac{\Delta_n}{n} - \frac{d - 2}{d}\big)$ for various values of $p$ on the Cayley graphs of (a) $\bbZ_2^{*4}$, (b) $\bbZ * \bbZ_2^{*2}$, (c) $\bbZ^{*2}$, grouped columnwise, based on $10^4$ Monte Carlo runs with $n = 10^6$ steps each. The critical probability here is $p_4 = 0.625$. The solid black lines depict the $N(0, \frac{4(d - 1)}{d^2})$ density.}
\label{fig:hist}
\end{figure}

In the course of proving Theorem~\ref{thm:rate}, we also obtain estimates on the probability of returning to $e$ at time $n$. This is known to be exponentially decaying in $n$ for the SRW (i.e. $p = 1/d$). Our bounds decay exponentially in $n$ for $0 \le p < 1/2$ (except for the special case $(p, d) = (0, 3)$). For $p \ge 1/2$, or $(p, d) = (0, 3)$, we only get stretched exponential decay in $n$.

Define
\[
    \alpha_{p, d} := \begin{cases}
        1 - \frac{2(1 - p)}{d - 1} & \text{if } 0 \le p \le 1/d, \\
        1 - 2p & \text{if } 1/d < p < 1/2.
    \end{cases}
\]
Notice that for any $d \ge 3$ and any $0 \le p < 1/2$, we have $\alpha_{p, d} > 0$, except in the special case $p = 0, d = 3$.

\begin{theorem}[Return probability estimates]\label{thm:return_to_zero}
    Suppose $0 \le p < 1/2$ and $(p, d) \ne (0, 3)$. Then
    \[
        \bbP(\Delta_n = 0) \le \exp\bigg(-\frac{n\alpha_{p, d}^2}{8}\bigg).
    \]
    Let $p \in [1/2, 1)$ or $(p, d) = (0, 3)$. Then there exists a constant $C_{p, d} > 0$, such that
    \[
        \bbP(\Delta_n = 0) \lesssim \exp\big(-C_{p, d} \cdot r_n\big).
    \]
\end{theorem}

\begin{remark}\label{rem:linear}
    The proof of Theorem~\ref{thm:return_to_zero} can be modified to produce similar bounds on $\bbP(\Delta_n \le cn)$ for a sufficiently small $c > 0$.
\end{remark}

\begin{remark}\label{rem:conc}
An inspection of the proof of Theorem~\ref{thm:return_to_zero} reveals that one would get exponential decay of the return probability $\bbP(\Delta_n = 0)$ for $p \ge 1/2$, or for the special case $(p, d) = (0, 3)$, if one is able to establish exponential concentration inequalities for $\Xi_n$ in these regimes.
\end{remark}

\begin{conjecture}
Prove exponential decay of the return probability in all regimes of $p$.
\end{conjecture}

We have shown that $r_n \big(\frac{\Delta_n}{n} - \frac{d - 2}{d}\big)$ is tight. We now present a partial result towards establishing the fluctuations of $\frac{\Delta_n}{n} - \frac{d - 2}{d}$.
\begin{proposition}\label{prop:fluc}
    Let $p \in [0, 1)$. Then
    \[
        \sqrt{n}\bigg[\frac{\Delta_n}{n} - \frac{d - 2}{d} - \frac{2(1 - pd)}{d - 1} \cdot \Xi_n \bigg] \convd N\bigg(0, \frac{4(d - 1)}{d^2}\bigg).
    \]
\end{proposition}

\begin{remark}
    For the SRW, i.e. $p = 1 / d$, Proposition~\ref{prop:fluc} reduces to the well-known central limit theorem: $\sqrt{n}(\frac{\Delta_n}{n} - \frac{d - 2}{d}) \convd N(0, \frac{4(d - 1)}{d^2})$.
\end{remark}

Proposition~\ref{prop:fluc} implies that in the critical and the supercritical regimes, one has
\[
    r_n \bigg(\frac{\Delta_n}{n} - \frac{d - 2}{d}\bigg) = \frac{2(1 - pd)}{d - 1} \cdot r_n \Xi_n + o_P(1).
\]
Therefore, the fluctuations of $r_n \big(\frac{\Delta_n}{n} - \frac{d - 2}{d}\big)$ in these regimes may be obtained by studying the fluctuations of $r_n \Xi_n$. This would presumably require sharp estimates on the number of times $\gamma_k = a$ in a sample path of the walk up to time $n$.

We leave the study of fluctuations as an open problem. 
\begin{conjecture}
    Establish distributional convergence of $r_n \big(\frac{\Delta_n}{n} - \frac{d - 2}{d}\big)$ and describe the limits.
\end{conjecture}

Our simulations suggest that the limits are Gaussian in the subcritical and critical regimes. In fact, in the subcritical regime the limit appears to be $N(0, \frac{4(d - 1)}{d^2})$, regardless of the value of $p$ (see Figure~\ref{fig:hist}).  In view of Proposition~\ref{prop:fluc}, this suggests that $\sqrt{n} \cdot \Xi_n = o_p(1)$ in the subcritical regime. Our simulations seem to corroborate this (see Figure~\ref{fig:plot-Xi}).

On the other hand, in the supercritical regime, these limits seem to depend on $p$ as well as the group $\Gamma$ (which seems quite plausible in view of Remark~\ref{rem:p=1}).

\begin{figure}[!t]
\centering
\begin{tabular}{ccc}
    \includegraphics[width = 0.315\textwidth]{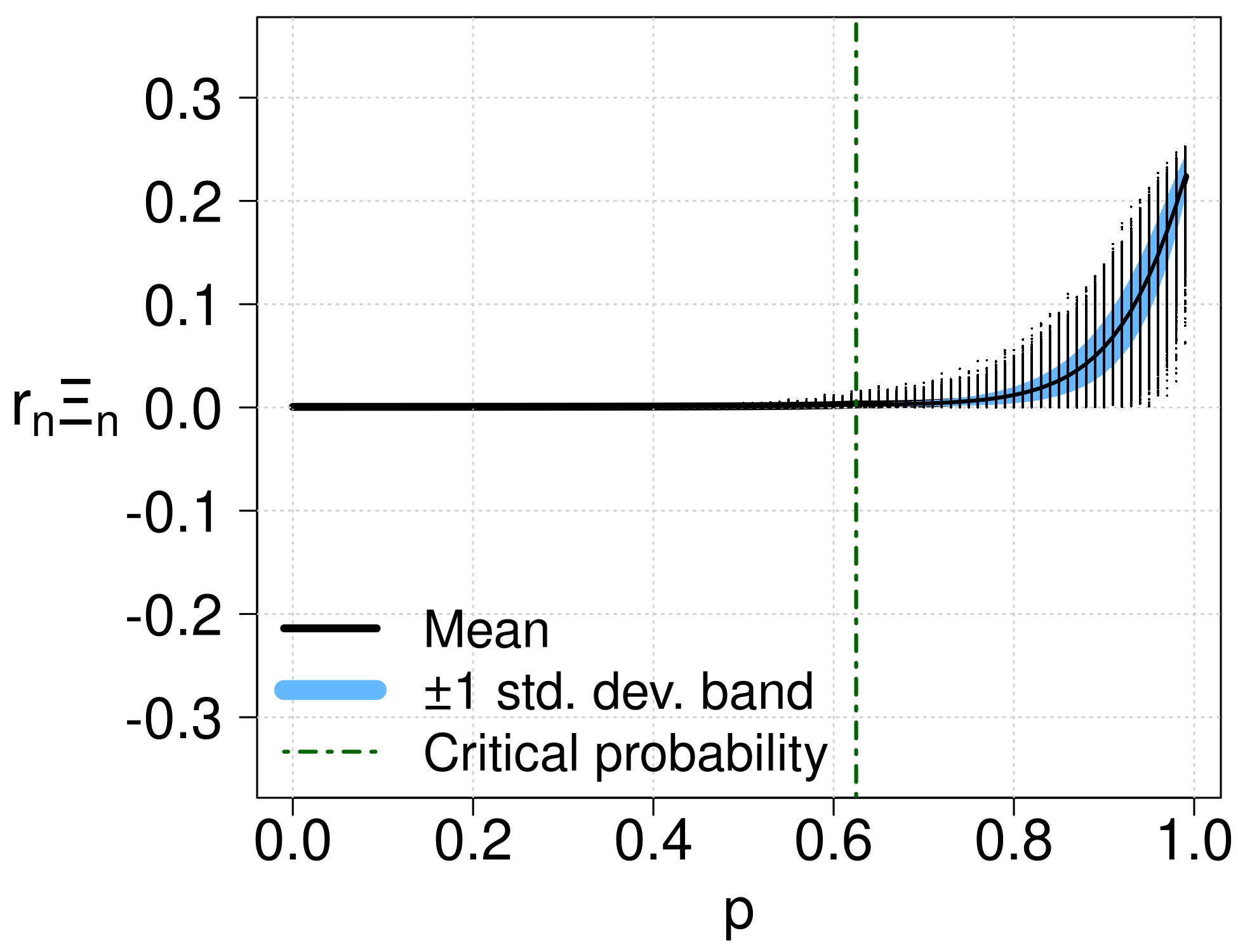} &
    \includegraphics[width = 0.315\textwidth]{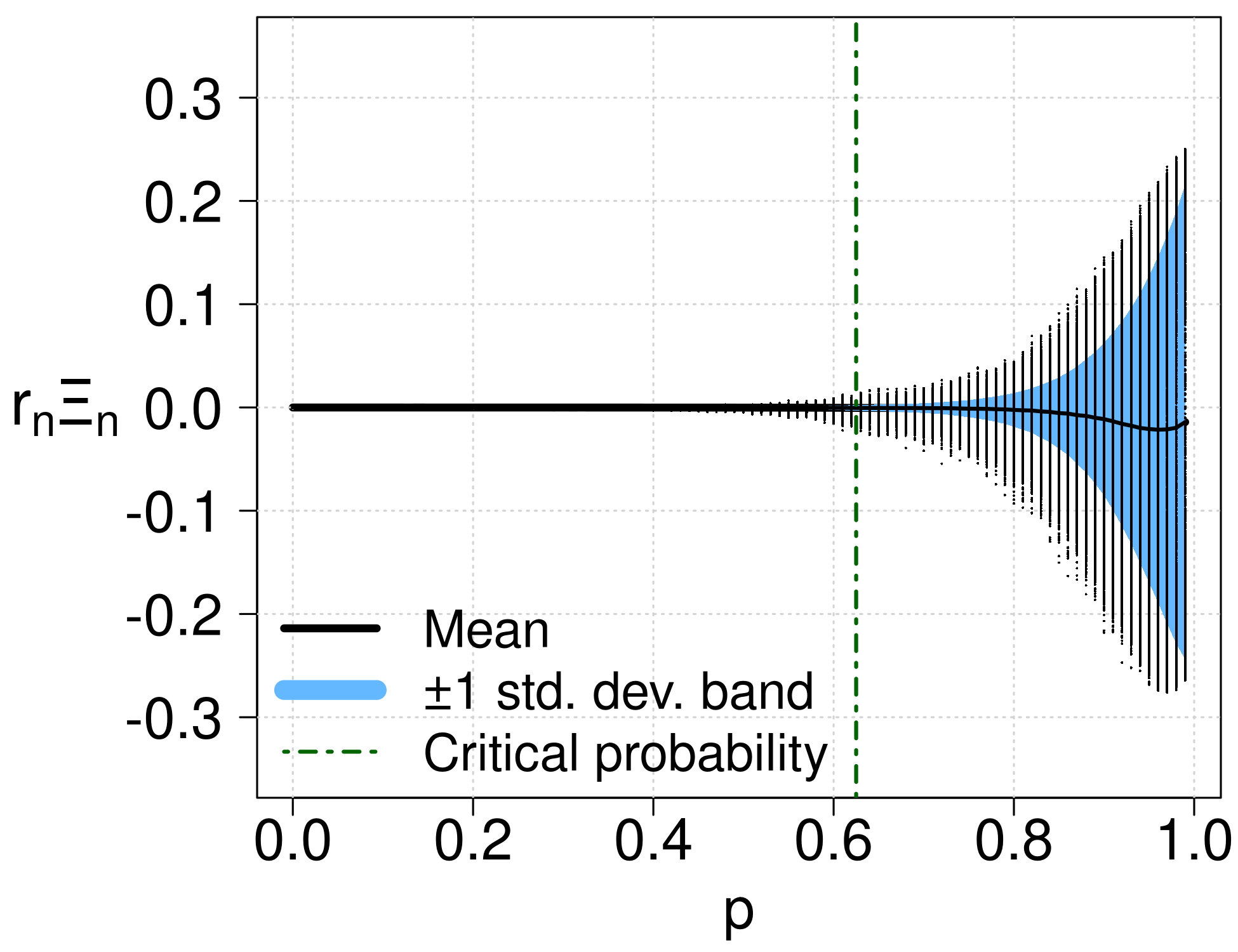} &
    \includegraphics[width = 0.315\textwidth]{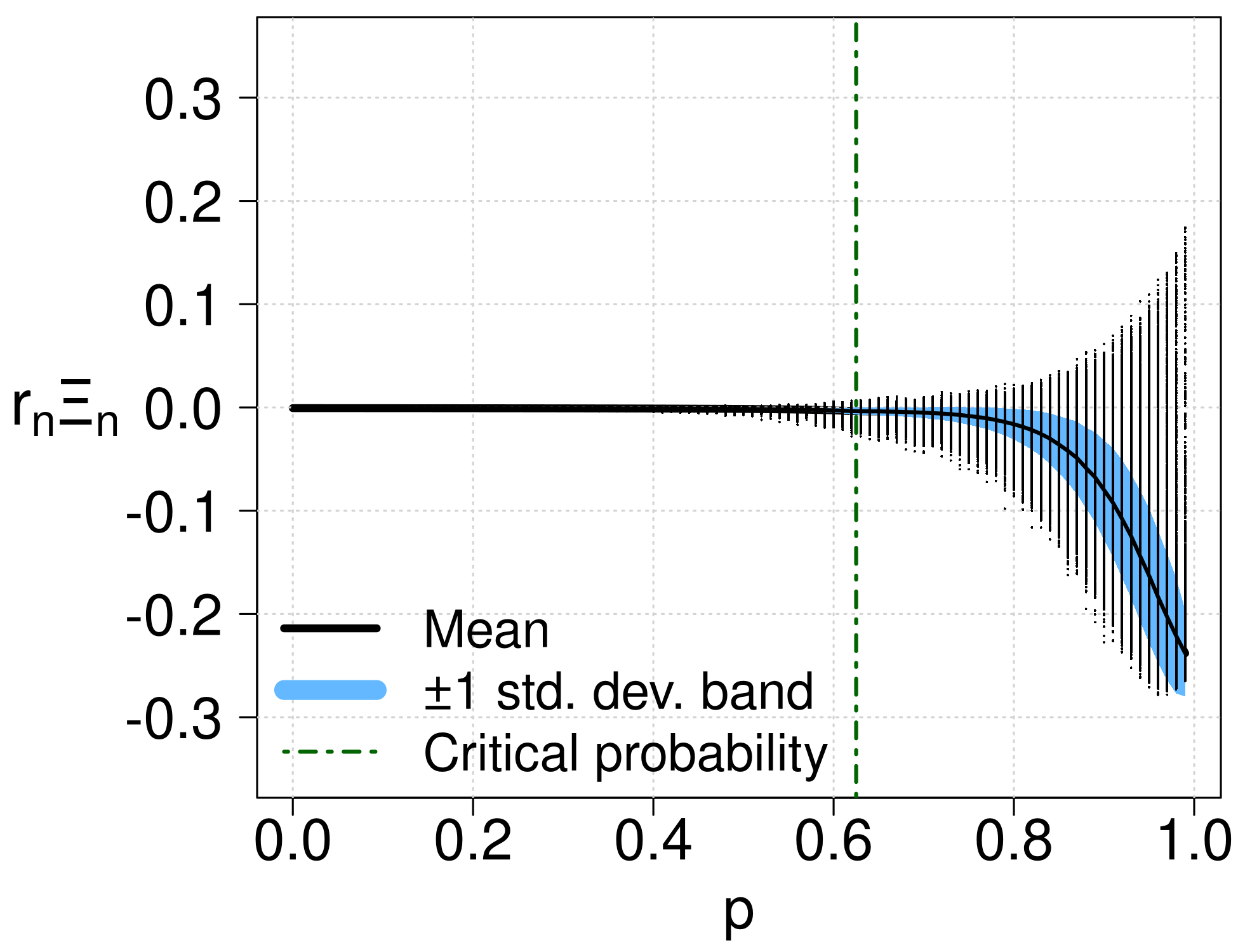} \\
    \hspace{2.25em}(a) & \hspace{2.25em}(b) & \hspace{2.25em}(c)
\end{tabular}
\caption{Scatter plot of $r_n \Xi_n$ for (a) $\bbZ_2^{*4}$, (b) $\bbZ * \bbZ_2^{*2}$, (c) $\bbZ^{*2}$, based on $10^4$ Monte Carlo runs with $n = 10^6$ steps each.}
\label{fig:plot-Xi}
\end{figure}

\section{Preliminaries}\label{sec:prelim}
Let $\cF_n$ denote the $\sigma$-algebra generated by steps up to time $n$. Note that if $w_n = e$, i.e. $\Delta_n = 0$, then
\[
    \Delta_{n + 1} - \Delta_n = 1.
\]
On the other hand, when $w_n \ne e$, i.e. $\Delta_n > 0$, we have $\bbP(g_D = \gamma_n \mid \cF_n) = \frac{\ell_n}{n}$. It is not difficult to see that given $\cF_n$, if $\Delta_n > 0$, then
\begin{equation}
    \Delta_{n + 1} - \Delta_n \sim \Rad(q_n)\footnote{Recall that a Rademacher random variable with parameter $q \in [0, 1]$ takes the value $+1$ with probability $q$ and $-1$ with probability $1 - q$.},
\end{equation}
where
\begin{align}\nonumber
    q_n &= p \cdot \dfrac{n - \ell_n}{n} + (1 - p) \cdot \dfrac{\ell_n}{n} + (1 - p) \cdot \dfrac{n - \ell_n}{n} \cdot \dfrac{d - 2}{d - 1} \\ \nonumber
        &= p + (1 - p) \cdot \frac{d - 2}{d - 1} + \bigg(1 - 2p - (1 - p) \cdot \frac{d - 2}{d - 1}\bigg) \cdot \frac{\ell_n}{n} \\
        &= 1 - \frac{1 - p}{d - 1} + \frac{1 - pd}{d - 1} \cdot \frac{\ell_n}{n}. \label{eq:q_n}
\end{align}
Even when $\Delta_n = 0$, we define $q_n$ by the above formula, so that $q_n = 1 - \frac{1 - p}{d - 1}$.

\subsection{The urn connection}
Consider the process of the $d$-vector of step counts $(N_n(a))_{a \in \cS}$. For any $a \in \cS$, we have
\begin{equation}\label{eq:urn-color-counts-recursion}
    N_{n + 1}(a) = N_n(a) + \ind(g_{n + 1} = a).
\end{equation}
Note that $g_{n + 1} = a$ if $g_D = a$ is chosen as the next step, or if $g_D \ne a$ and step $a$ is taken instead of $g_D$. Thus
\begin{equation}\label{eq:urn-next-color-prob}
    \bbP(g_{n + 1} = a \mid \cF_n) = \frac{N_n(a)}{n} \cdot p + \frac{n - N_n(a)}{n} \cdot \frac{1 - p}{d - 1}.
\end{equation}
Therefore the vector of step counts $(N_n(a))_{a \in \cS}$ is an urn process (a fact first observed by \cite{baur2016elephant}) with $d$ colors and mean replacement matrix
\[
    P = \begin{pmatrix}
        p & \frac{1 - p}{d - 1} & \cdots & \frac{1 - p}{d - 1} \\
        \frac{1 - p}{d - 1} & p & \cdots & \frac{1 - p}{d - 1} \\
        \vdots & \vdots & \ddots & \vdots \\
        \frac{1 - p}{d - 1} & \frac{1 - p}{d - 1} & \cdots & p \\
    \end{pmatrix}_{d \times d}.
\]
Note that $P$ is a symmetric matrix with Perron-Frobenius eigenvalue $\lambda_1 = 1$ with eigenvector $\bone_d$ and second eigenvalue $\lambda_2 = \frac{pd - 1}{d - 1}$, with multiplicity $d - 1$. Notice that $p = \frac{d + 1}{2d}$ is equivalent to the condition $\lambda_2 = \frac{1}{2}\lambda_1$, which defines the critical regime for the urn process described above. This is precisely how the critical probability $p_d = \frac{d + 1}{2d}$ arises. Using the embedding of the urn process into a continuous-time multi-type branching process (due to \cite{athreya1972branching}), one may prove that for any $a \in \cS$,
\begin{equation}\label{eq:urn-asconv-proportions}
    \frac{N_n(a)}{n} - \frac{1}{d} \convas 0.
\end{equation}
See, e.g., Theorem~3.21 of \cite{janson2004functional}, who also obtained fluctuations of these proportions (see Theorems~3.22-3.24 of his paper). In fact, one may also prove the following estimates on the rate of convergence.
\begin{lemma}\label{lem:urn-moments-of-proportions}
    Let $p \in [0, 1)$. For any $m \ge 1$, there exists a constant $C'_{p, d, m} > 0$ such that for any $a \in \cS$, one has
    \[
        \bbE \bigg|\frac{N_n(a)}{n} - \frac{1}{d}\bigg|^m \le (C'_{p, d, m})^m \cdot r_n^{-m}.
    \]
    In fact, for any $m \ge 2$, one may take $C'_{p, d, m} = C'_{p, d} \cdot m$ for some constant $C'_{p, d} > 0$.
\end{lemma}
For the MERW on $\bbZ^d$, a result of this form was worked out in Lemma~4.3 of \cite{qin2025recurrence} for the moments of $\big|\frac{N_n(\be_i) + N_n(-\be_i)}{n} - \frac{1}{d}\big|$. The proof utilises the Burkholder-Davis-Gundy inequality \cite{davis1970intergrability}.  Since that argument works here almost verbatim, we omit the proof. We are explicitly specifying the dependence of $C'_{p, d, m}$ on $m$ (see, e.g., \cite[Chapter 11]{chow2003probability}) as we shall choose an optimal $m$ in a later argument.

\subsection{A concentration result for \texorpdfstring{$\Xi_n$}{Xi-n}.}
As mentioned before, a key ingredient in our analysis will be concentration of $\Xi_n$. We now derive a basic concentration result using the urn connection.

We shall need the following estimates on Ces\`{a}ro averages, obtained easily via integral comparison.
\begin{lemma}\label{lem:cesaro_ests}
Let $\alpha \ge 0$ and $\beta > 0$. We have
\[
    \frac{1}{n}\sum_{k = 1}^n \frac{(\log k)^{\alpha}}{k^{\beta}} \asymp
    \begin{cases}
        \frac{(\log n)^{\alpha}}{n^{\beta}} & \text{if } 0 < \beta < 1, \\
        \frac{(\log n)^{\alpha + 1}}{n^{\beta}} & \text{if } \beta = 1, \\
        \frac{1}{n} & \text{if } \beta > 1.
    \end{cases}
\]
\end{lemma}
The following concentration result for $\Xi_n$ is an easy consequence of Lemmas~\ref{lem:urn-moments-of-proportions} and \ref{lem:cesaro_ests}.
\begin{lemma}\label{lem:cesaro_ell}
    Let $p \in [0, 1)$ and $m \ge 1$. There is a constant $\tilde{C}_{p, d, m} > 0$, such that
    \[
        \bbE|\Xi_n|^m \le (\tilde{C}_{p, d, m})^m \cdot r_n^{-m}.
    \]
    In fact, for any $m \ge 2$, one may take $\tilde{C}_{p, d, m} = \tilde{C}_{p, d} \cdot m$ for some constant $\tilde{C}_{p, d} > 0$. Consequently, for any $u \ge \frac{2\tilde{C}_{p, d}\exp(1)}{r_n}$, one has
    \[
        \bbP(|\Xi_n| \ge u) \le \exp\bigg(-\frac{ur_n}{\tilde{C}_{p, d}\exp(1)}\bigg).
    \]
\end{lemma}

\begin{proof}[Proof of Lemma~\ref{lem:cesaro_ell}]
We have
\begin{align*}
    \bigg|\frac{1}{n}\sum_{k = 0}^{n - 1}\bigg(\frac{\ell_k}{k} - \frac{\ind(\Delta_k > 0)}{d}\bigg)\bigg| &\le \frac{1}{n}\sum_{k = 0}^{n - 1}\bigg|\frac{\ell_k}{k} - \frac{\ind(\Delta_k > 0)}{d}\bigg| \\
                                                                                                   &\le \frac{1}{n} \sum_{k = 1}^{n - 1} \sum_{a \in \cS} \bigg|\frac{N_k(a)}{k} - \frac{1}{d}\bigg| \\
                                                                                                   &= \sum_{a \in \cS} \frac{1}{n} \sum_{k = 1}^{n - 1} \bigg|\frac{N_k(a)}{k} - \frac{1}{d}\bigg|.
\end{align*}
Applying Minkowski's inequality and Lemma~\ref{lem:urn-moments-of-proportions}, we have
\begin{align*}
    \bigg[\bbE\bigg|\frac{1}{n}\sum_{k = 0}^{n - 1}\bigg(\frac{\ell_k}{k} - \frac{\ind(\Delta_k > 0)}{d}\bigg)\bigg|^m\bigg]^{\frac{1}{m}} &= \sum_{a \in \cS} \frac{1}{n} \sum_{k = 1}^{n - 1} \bigg[\bbE\bigg|\frac{N_k(a)}{k} - \frac{1}{d}\bigg|^m\bigg]^{\frac{1}{m}} \\
                                                                                                                                           &\le C_{p, d, m} \cdot \sum_{a \in \cS} \frac{1}{n} \sum_{k = 1}^{n - 1} r_k^{-1}.
\end{align*}
By applying Lemma~\ref{lem:cesaro_ests} to the right-hand side, we conclude that for some constant $\tilde{C}_{p, d, m} > 0$, one has
\[
    \big[\bbE|\Xi_n|^m\big]^{\frac{1}{m}} \le \tilde{C}_{p, d, m} \cdot r_n^{-1}.
\]
It is clear from the above proof that $\tilde{C}_{p, d, m} = C'_{p, d, m} \cdot C''_{p, d}$ for some constant $C''_{p, d} > 0$. As such for any $m \ge 2$, we may take $\tilde{C}_{p, d, m} = \tilde{C}_{p, d} \cdot m$ for some constant $\tilde{C}_{p, d} > 0$.

Now, with $m \ge 2$, 
\begin{align*}
    \bbP(|\Xi_n| \ge u) &\le \frac{\bbE|\Xi_n|^m}{u^m} \le \bigg(\frac{\tilde{C}_{p, d} \cdot m}{u r_n}\bigg)^{m} \le \sup_{t \ge 2} \bigg(\frac{\tilde{C}_{p, d} \cdot t}{u r_n}\bigg)^{t}.
\end{align*}
It is easy to check that the supremum is attained at $t_* = \frac{u r_n}{\tilde{C}_{p, d}\exp(1)} \ge 2$, which gives
\[
    \bbP(|\Xi_n| \ge u) \le \exp\bigg(-\frac{u r_n}{C_{p, d}\exp(1)}\bigg).
\]
This completes the proof.
\end{proof}

\section{Proofs of our main results}\label{sec:proofs}
As discussed in the introduction, the non-commutative nature of random walks on groups necessitates proofs with a markedly different flavour from the martingale-based techniques used to analyse elephant random walks on $\mathbb{Z}^d$.
\begin{proof}[Proof of Theorem~\ref{thm:LLN}]
We have
\begin{align*}
    G_n &:= \bbE [\Delta_{n + 1} - \Delta_n \mid \cF_n] \\
    &= (2 q_n - 1) \ind(\Delta_n > 0) + \ind(\Delta_n = 0) \\
                   &= \bigg(1 - \frac{2(1 - p)}{d - 1} + \frac{2(1 - pd)}{d - 1} \cdot \frac{\ell_n}{n}\bigg) \ind(\Delta_n > 0) + \ind(\Delta_n = 0) \\
                   &= 1 - \frac{2(1 - p)}{d - 1} \cdot \ind(\Delta_n > 0) + \frac{2(1 - pd)}{d - 1} \cdot \frac{\ell_n}{n}.
\end{align*}
Therefore $Y_{n + 1} := \Delta_{n + 1} - \Delta_n - G_n$
forms a (bounded) martingale difference sequence. It follows from the strong law of large numbers (SLLN) for martingales that
\[
    \frac{1}{n} \sum_{k = 1}^{n} Y_k \convas 0.
\]
In other words,
\begin{equation}\label{eq:as_conv_diff}
    \frac{\Delta_n}{n} - \bigg[ 1 - \frac{2(1 - p)}{d - 1} \cdot \frac{1}{n} \sum_{k = 0}^{n - 1} \ind(\Delta_k > 0) + \frac{2(1 - pd)}{d - 1} \cdot \frac{1}{n}\sum_{k = 0}^{n - 1} \frac{\ell_k}{k} \bigg] \convas 0.
\end{equation}
Notice that
\begin{equation}\label{eq:as_conv_ell}
    \bigg|\frac{\ell_n}{n} - \frac{\ind(\Delta_n > 0)}{d}\bigg| = \bigg|\frac{N_n(\gamma_n)}{n} - \frac{1}{d}\bigg| \ind(\Delta_n > 0) \le \max_{a \in \cS}\bigg|\frac{N_n(a)}{n} - \frac{1}{d}\bigg| \convas 0,
\end{equation}
where the last assertion follows from \eqref{eq:urn-asconv-proportions}.
It then follows that the Ces\`{a}ro average
\begin{equation}\label{eq:as_conv_ell_cesaro}
    \frac{1}{n}\sum_{k = 0}^{n - 1} \frac{\ell_k}{k} - \frac{1}{d} \cdot \frac{1}{n} \sum_{k = 0}^{n - 1} \ind(\Delta_k > 0) \convas 0. 
\end{equation}
Combining \eqref{eq:as_conv_diff} and \eqref{eq:as_conv_ell_cesaro}, we have
\[
    \frac{\Delta_n}{n} - \bigg[ 1 - \frac{2(1 - p)}{d - 1} \cdot \frac{1}{n} \sum_{k = 0}^{n - 1} \ind(\Delta_k > 0) + \frac{2(1 - pd)}{d - 1} \cdot \frac{1}{d} \cdot \frac{1}{n} \sum_{k = 0}^{n - 1} \ind(\Delta_k > 0) \bigg] \convas 0.
\]
Now,
\begin{align}\label{eq:simplification}\nonumber
     & 1 - \frac{2(1 - p)}{d - 1} \cdot \frac{1}{n} \sum_{k = 0}^{n - 1} \ind(\Delta_k > 0) + \frac{2(1 - pd)}{d - 1} \cdot \frac{1}{d} \cdot \frac{1}{n} \sum_{k = 0}^{n - 1} \ind(\Delta_k > 0) \\ \nonumber
     &= 1 - \bigg(\frac{2(1 - p)}{d - 1} - \frac{2(1 - pd)}{d - 1} \cdot \frac{1}{d}\bigg) \cdot \frac{1}{n} \sum_{k = 0}^{n - 1} \ind(\Delta_k > 0) \\
     &= 1 - \frac{2}{d} \cdot \frac{1}{n} \sum_{k = 0}^{n - 1} \ind(\Delta_k > 0) \\ \nonumber
     &\ge 1 - \frac{2}{d}.
\end{align}
It follows that
 \[
    \liminf_{n \to \infty} \frac{\Delta_n}{n} \ge 1 - \frac{2}{d} > 0 \quad \text{a.s.},
 \]
so that the local time at the identity
\begin{equation}\label{eq:local_time_at_root}
    \sum_{n \ge 1} \ind(\Delta_n = 0) < \infty \quad \text{ a.s.}
\end{equation}
Hence, the Ces\`{a}ro average
\[
    \frac{1}{n} \sum_{k = 0}^{n - 1} \ind(\Delta_k > 0) \convas 1,
\]
which in turn implies by virtue of \eqref{eq:simplification} that
\[
    \frac{\Delta_n}{n} \convas 1 - \frac{2}{d}.
\]
This completes the proof.
\end{proof}

\begin{proof}[Proof of Theorem~\ref{thm:return_to_zero}]
We recall that, given $\cF_n$, if $w_n \ne e$, then
\[
    \Delta_{n + 1} - \Delta_{n} \sim \Rad(q_n),
\]
where $q_n$ is as in \eqref{eq:q_n}. Otherwise, $\Delta_{n + 1} - \Delta_n = 1$. We now couple $\Delta_n$ to another random sequence that eliminates this reflective behaviour at $e$. For $n \ge 1$ and given $\cF_{n - 1}$, let
\[
    Z_n = \begin{cases}
        \Delta_n - \Delta_{n - 1} & \text{if } w_{n - 1} \ne e, \\
        B_n & \text{otherwise,}
    \end{cases}
\]
where $B_n \sim \Rad(q_{n - 1})$ is sampled independently of $g_n$. Since $Z_n \le \Delta_n - \Delta_{n - 1}$ by construction, we have
\begin{equation}\label{eq:return_prob_initial_bound}
    \bbP(\Delta_n = 0) \le \bbP\bigg(\sum_{k = 1}^n Z_k \le 0\bigg) = \bbP\bigg(\frac{1}{n}\sum_{k = 1}^n Z_k \le 0\bigg).
\end{equation}
Therefore it is enough to prove the desired bounds on $\bbP\big(\frac{1}{n}\sum_{k = 1}^n Z_k \le 0\big)$.

Note that
\[
    \bbE[Z_n \mid \cF_{n - 1}] = 2 q_{n - 1} - 1.
\]
Consider the martingale difference sequence $W_n = Z_n - (2q_{n - 1} - 1)$. Clearly, $|W_n| \le 2$. By Azuma's inequality \cite{azuma1967weighted}, for any $\varepsilon > 0$,
\[
    \bbP\bigg(\sum_{k = 1}^n W_k \le -\varepsilon\bigg) \le \exp\bigg(-\frac{\varepsilon^2}{8n}\bigg).
\]
A fortiori,
\begin{equation}\label{eq:Azuma-consequence}
    \bbP\bigg(\frac{1}{n}\sum_{k = 1}^n Z_k \le \frac{1}{n}\sum_{k = 1}^n (2 q_{k - 1} - 1) - \varepsilon\bigg) \le \exp\bigg(-\frac{n\varepsilon^2}{8}\bigg).
\end{equation}
Then, we may write
\begin{align}\label{eq:decomp-mean_Z}\nonumber
    \bbP\bigg(\frac{1}{n}\sum_{k = 1}^n Z_k \le 0\bigg) &= \bbP\bigg(\frac{1}{n}\sum_{k = 1}^n Z_k \le 0, \frac{1}{n}\sum_{k = 1}^n (2 q_{k - 1} - 1) \ge \varepsilon\bigg) \\ \nonumber
                                                        &\qquad\qquad\qquad + \bbP\bigg(\frac{1}{n}\sum_{k = 1}^n Z_k \le 0, \frac{1}{n}\sum_{k = 1}^n (2 q_{k - 1} - 1) < \varepsilon\bigg) \\ \nonumber
                                                        &\le\bbP\bigg(\frac{1}{n}\sum_{k = 1}^n Z_k \le \frac{1}{n}\sum_{k = 1}^n (2 q_{k - 1} - 1) - \varepsilon\bigg) + \bbP\bigg(\frac{1}{n}\sum_{k = 1}^n (2 q_{k - 1} - 1) < \varepsilon\bigg) \\
                       &\le \exp\bigg(-\frac{n\varepsilon^2}{8}\bigg) + \bbP\bigg(\frac{1}{n}\sum_{k = 1}^n (2 q_{k - 1} - 1) < \varepsilon\bigg).
\end{align}
It therefore boils down to estimating $\bbP\big(\frac{1}{n}\sum_{k = 1}^n (2 q_{k - 1} - 1) < \varepsilon\big)$. Towards that end, note that
\begin{equation}\label{eq:mean_incremenet_cesaro_expression}
    \frac{1}{n}\sum_{k = 1}^n (2 q_{k - 1} - 1) = 1 - \frac{2(1 - p)}{d - 1} + \frac{2(1 - pd)}{d - 1} \cdot \frac{1}{n}\sum_{k = 0}^{n-1}\frac{\ell_k}{k}.
\end{equation}

\textbf{Case I ($\boldsymbol{0 \le p < 1/2}$, $\boldsymbol{(p, d) \ne (0, 3)}$).}
We claim that in this regime
\[
    \frac{1}{n}\sum_{k = 1}^n (2 q_{k - 1} - 1) \ge \alpha_{p, d}.
\]
Indeed, if $0 \le p \le 1/d$, then ignoring the last term in \eqref{eq:mean_incremenet_cesaro_expression}, we get
\[
    \frac{1}{n}\sum_{k = 1}^n (2 q_{k - 1} - 1) \ge 1 - \frac{2(1 - p)}{d - 1}.
\]                                               
On the other hand, if $1/d < p \le 1/2$, we have $(1 - pd) < 0$. Therefore, since $\frac{1}{n}\sum_{k = 0}^{n-1}\frac{\ell_k}{k}\le 1$, we get from \eqref{eq:mean_incremenet_cesaro_expression} that
\begin{align*}
    \frac{1}{n}\sum_{k = 1}^n (2 q_{k - 1} - 1) &\ge 1 - \frac{2(1 - p)}{d - 1} + \frac{2(1 - pd)}{d - 1} \\
                                                &= 1 - 2p.
\end{align*}                                   
Therefore, in the regime $0 \le p < 1/2$, $(p, d) \ne (0, 3)$, we get from \eqref{eq:decomp-mean_Z} with the choice $\varepsilon = \alpha_{p, d}$ that
\begin{align*}
    \bbP\bigg(\frac{1}{n}\sum_{k = 1}^n Z_k \le 0\bigg) \le \exp\bigg(-\frac{n\alpha_{p, d}^2}{8}\bigg).
\end{align*}

\textbf{Case II ($\boldsymbol{p \ge 1/2}$).}
For this regime, we shall rewrite $\frac{1}{n}\sum_{k = 1}^n (2q_{k - 1} - 1)$ as follows:
\begin{align*}
    \frac{1}{n}\sum_{k = 1}^n (2 q_{k - 1} - 1) = \bigg(1 - \frac{2}{d}\bigg) + \frac{2(1 - pd)}{d - 1} \cdot \Xi_n - \frac{2(1 - pd)}{d - 1} \cdot \frac{1}{n}\sum_{k = 0}^{n-1} \frac{\ind(\Delta_k = 0)}{d}.
\end{align*}
Since $1 - pd < 0$ in this regime, we may drop the last term to get
\begin{equation}\label{eq:lower_bd_qk}
    \frac{1}{n}\sum_{k = 1}^n (2 q_{k - 1} - 1) \ge \bigg(1 - \frac{2}{d}\bigg) + \frac{2(1 - pd)}{d - 1} \cdot \Xi_n.
\end{equation}
Now for any $0 < \varepsilon < 1 - \frac{2}{d}$, using \eqref{eq:lower_bd_qk}, we have
\begin{align*}
    \bbP\bigg(\frac{1}{n}\sum_{k = 1}^n (2 q_{k - 1} - 1) < \varepsilon\bigg) &\le \bbP\bigg(\frac{1}{n}\sum_{k = 1}^n (2 q_{k - 1} - 1) - \bigg(1 - \frac{2}{d}\bigg) < -\bigg(1 - \frac{2}{d} - \varepsilon\bigg)\bigg) \\
                                                                              &\le \bbP\bigg(\frac{2(1 - pd)}{d - 1} \cdot \Xi_n < -\bigg(1 - \frac{2}{d} - \varepsilon\bigg)\bigg) \\
                                                                              &\le \bbP\bigg(|\Xi_n| > \frac{1 - \frac{2}{d} - \varepsilon}{\frac{2(pd - 1)}{d - 1}}\bigg) \\
                                                                              &\le \exp\big(-C_{p, d, \varepsilon} \cdot r_n \big),
\end{align*}
where in the last line we have used Lemma~\ref{lem:cesaro_ell}, and $C_{p, d, \varepsilon} = \frac{1 - \frac{2}{d} - \varepsilon}{\frac{2(pd - 1)}{d - 1}\tilde{C}_{p, d} \exp(1)}$. Choosing a fixed $\varepsilon \in (0, 1 - \frac{2}{d})$, it now follows from \eqref{eq:decomp-mean_Z} that
\[
    \bbP\bigg(\frac{1}{n}\sum_{k = 1}^n Z_k \le 0\bigg) \lesssim \exp(- C_{p, d} \cdot r_n),
\]
for some constant $C_{p, d} > 0$.

\textbf{Case III ($\boldsymbol{(p, d) = (0, 3)}$).} For this special case, we make use of the trivial bound
\[
    \sum_{k = 0}^{n - 1} \ind(\Delta_k = 0) \le \bigg\lfloor \frac{n + 1}{2} \bigg\rfloor,
\]
which is saturated when the walk alternates between $e$ and a neighbour of $e$. Therefore
\[
    \frac{1}{n} \sum_{k = 0}^{n - 1} \ind(\Delta_k = 0) \le \frac{1}{2} + \frac{1}{2n},
\]
or, equivalently,
\[
    \frac{1}{n} \sum_{k = 0}^{n - 1} \ind(\Delta_k > 0) \ge \frac{1}{2} - \frac{1}{2n}.
\]
Now in this special case, we see from \eqref{eq:mean_incremenet_cesaro_expression} that
\begin{align*}
    \frac{1}{n} \sum_{k = 1}^n (2 q_{k - 1} - 1) &= \frac{1}{n} \sum_{k = 0}^{n - 1} \frac{\ell_k}{k} \\
                                                 &= \Xi_n + \frac{1}{n} \sum_{k = 0}^{n - 1} \frac{\ind(\Delta_k > 0)}{3} \\
                                                 &\ge \Xi_n + \frac{1}{6} - \frac{1}{6n}.
\end{align*}
Hence, for $\varepsilon \in (0, 1/6)$ and $n > \frac{1}{1 - 6\varepsilon}$, we have
\begin{align*}
    \bbP\bigg(\frac{1}{n}\sum_{k = 1}^n (2 q_{k - 1} - 1) < \varepsilon\bigg) &\le \bbP\bigg(\Xi_n + \frac{1}{6} - \frac{1}{6n} < \varepsilon\bigg) \\
                                                                              &\le \bbP\bigg(|\Xi_n| > \frac{1}{6} - \frac{1}{6n} - \varepsilon\bigg) \\
                                                                              &\lesssim \exp\big(-C_{0, 3, \varepsilon} \cdot r_n \big),
\end{align*}
where in the last line we have used Lemma~\ref{lem:cesaro_ell}, and $C_{0, 3, \varepsilon} = \frac{\frac{1}{7} - \varepsilon}{\tilde{C}_{0, 3} \exp(1)}$. Fix an $\varepsilon \in (0, \frac{1}{7})$. As in Case II, it follows from \eqref{eq:decomp-mean_Z} that
\[
    \bbP\bigg(\frac{1}{n}\sum_{k = 1}^n Z_k \le 0\bigg) \lesssim \exp(- C_{0, 3} \cdot r_n),
\]
for some constant $C_{0, 3} > 0$.
\end{proof}

\begin{proof}[Proof of Theorem~\ref{thm:rate}]
We begin with the decomposition
\begin{equation}\label{eq:speed_decomp}
    \frac{\Delta_n}{n} - \frac{d - 2}{d} = \frac{M_n}{n} + \frac{2}{d} \cdot \frac{1}{n} \sum_{k = 0}^{n - 1} \ind(\Delta_k = 0) + \frac{2(1 - pd)}{d - 1} \cdot \Xi_n,
\end{equation}
where, as before, $M_n = \sum_{k = 1}^n Y_k$. For any $m \ge 1$, by Minkowski's inequality, we have
\begin{align}\label{eq:three_parts_upper} \nonumber
    \bigg[\bbE \bigg|\frac{\Delta_n}{n} &- \frac{d - 2}{d}\bigg|^m\bigg]^{\frac{1}{m}} \\ \nonumber
                                        &\le \frac{\big[\bbE|M_n|^m\big]^{\frac{1}{m}}}{n} + \frac{2}{d} \cdot \bigg[\bbE \bigg(\frac{1}{n} \sum_{k = 0}^{n - 1} \ind(\Delta_k = 0)\bigg)^m\bigg]^{\frac{1}{m}}  + \bigg|\frac{2(1 - pd)}{d - 1}\bigg| \cdot \big[\bbE |\Xi_n|^m\big]^{\frac{1}{m}} \\
                                        &\lesssim \frac{\big[\bbE|M_n|^m\big]^{\frac{1}{m}}}{n} + \bigg[\frac{1}{n} \sum_{k = 0}^{n - 1} \bbP(\Delta_k = 0) \bigg]^{\frac{1}{m}} + \big[\bbE |\Xi_n|^m\big]^{\frac{1}{m}}.
\end{align}
On the other hand, rewriting \eqref{eq:speed_decomp} as
\[
    \frac{M_n}{n} = \frac{\Delta_n}{n} - \frac{d - 2}{d} - \frac{2}{d} \cdot \frac{1}{n} \sum_{k = 0}^{n - 1} \ind(\Delta_k = 0) - \frac{2(1 - pd)}{d - 1} \cdot \Xi_n,
\]
we may similarly obtain
\[
    \frac{\big[\bbE|M_n|^m\big]^{\frac{1}{m}}}{n} \le \bigg[\bbE\bigg|\frac{\Delta_n}{n} - \frac{d - 2}{d}\bigg|^m\bigg]^{\frac{1}{m}} + \frac{2}{d} \cdot \bigg[\frac{1}{n} \sum_{k = 0}^{n - 1} \bbP(\Delta_k = 0)\bigg]^{\frac{1}{m}} + \bigg|\frac{2(1 - pd)}{d - 1}\bigg| \cdot \big[\bbE|\Xi_n|^m \big]^{\frac{1}{m}},
\]
or, equivalently,
\begin{equation}\label{eq:three_parts_lower}
    \bigg[\bbE\bigg|\frac{\Delta_n}{n} - \frac{d - 2}{d}\bigg|^m\bigg]^{\frac{1}{m}} \ge \frac{\big[\bbE|M_n|^m\big]^{\frac{1}{m}}}{n} - \frac{2}{d} \cdot \bigg[\frac{1}{n} \sum_{k = 0}^{n - 1} \bbP(\Delta_k = 0)\bigg]^{\frac{1}{m}} - \bigg|\frac{2(1 - pd)}{d - 1}\bigg| \cdot \big[\bbE|\Xi_n|^m\big]^{\frac{1}{m}}.
\end{equation}

In view of \eqref{eq:three_parts_upper} and \eqref{eq:three_parts_lower}, to prove the desired upper and lower bounds, we need to pinpoint the growth rate of $\bbE|M_n|^m$ and give upper bounds to $\frac{1}{n}\sum_{k = 0}^{n - 1} \bbP(\Delta_k = 0)$ and $\bbE|\Xi_n|^m$. 

Lemma~\ref{lem:cesaro_ell} provides upper bounds on $\bbE|\Xi_n|^m$, and using the return probability estimates from Theorem~\ref{thm:return_to_zero}, we have
\[
    \frac{1}{n}\sum_{k = 0}^{n - 1}\bbP(\Delta_k = 0) \lesssim \frac{1}{n}.
\]

We shall now estimate $\bbE|M_n|^m$. The Burkholder-Davis-Gundy inequality gives
\begin{equation}\label{eq:BDJ}
     \bbE \bigg(\sum_{k = 1}^n Y_k^2\bigg)^{m/2} \lesssim \bbE |M_n|^m \lesssim \bbE \bigg(\sum_{k = 1}^n Y_k^2\bigg)^{m/2}.
 \end{equation}
Now the SLLN for the bounded martingale difference sequence $(Y_n^2 - \bbE[Y_n^2 \mid \cF_{n - 1}])_{n \ge 1}$ yields
\[
    \frac{1}{n}\sum_{k = 1}^n Y_k^2 - \frac{\langle M \rangle_n}{n} = \frac{1}{n}\sum_{k = 1}^n (Y_k^2 - \bbE[Y_k^2 \mid \cF_{k - 1}]) \convas 0.
\]
But, as shown in \eqref{eq:conv_quad_var} (in the proof of Proposition~\ref{prop:fluc}), $\frac{\langle M \rangle_n}{n} \convas \frac{4(d - 1)}{d^2}$. Hence
\[
    \frac{1}{n}\sum_{k = 1}^n Y_k^2 \convas \frac{4(d - 1)}{d^2}.
\]
Since the $Y_k$'s are bounded, for any $m \ge 1$, 
\[
    \bbE\bigg|\frac{1}{n}\sum_{k = 1}^n Y_k^2 - \frac{4(d - 1)}{d^2}\bigg|^m \to 0
\]
by dominated convergence. In particular, this implies that
\[
    \bbE\bigg(\sum_{k = 1}^n Y_k^2\bigg)^m \asymp n^m.
\]
Plugging this into \eqref{eq:BDJ}, we see that for any $m \ge 1$,
\[
    \bbE[|M_n|^m] \asymp n^{m / 2}.
\]

Plugging our estimates on $\bbE|M_n|^m$, $\frac{1}{n}\sum_{k = 0}^{n - 1}\bbP(\Delta_k = 0)$ and $\bbE|\Xi_n|^m$ into \eqref{eq:three_parts_upper}, we obtain the stated upper bounds. For the lower bound \eqref{eq:lower-bd}, we have to choose $p$ sufficiently close to $1/d$ (depending on $m$, as the implicit constants in the Burkholder-Davis-Gundy bounds depend on $m$) such that the first term in \eqref{eq:three_parts_lower} dominates the third term (both are $n^{-m/2}$ in order). This completes the proof.
\end{proof}

\begin{proof}[Proof of Proposition~\ref{prop:fluc}]
The proof is an application of the martingale central limit theorem to $M_n = \sum_{k = 1}^n Y_k$. We shall first compute the quadratic variation of $M_n$. Note that
\begin{align*}
    \bbE[Y_k^2 \mid \cF_{k - 1}] &= \bbE[(\Delta_{k} - \Delta_{k - 1} - G_{k - 1})^2 \mid \cF_{k - 1}] \\
                                 &= \bbE[(\Delta_{k} - \Delta_{k - 1})^2 \mid \cF_{k - 1}] - G_{k - 1}^2 \\
                                                &= 1 - G_{k - 1}^2.
\end{align*}
Thus
\begin{align*}
    \langle M \rangle_n = \sum_{k = 1}^n \bbE[Y_k^2 \mid \cF_{k - 1}] &= n - \sum_{k = 0}^{n - 1} G_k^2 = n \bigg[1 - \frac{1}{n}\sum_{k = 0}^{n - 1} G_k^2\bigg].
\end{align*}
Now, using \eqref{eq:as_conv_ell} and \eqref{eq:local_time_at_root}, we see that $G_n^2 \convas \big(1 - \frac{2}{d}\big)^2$, and therefore
\begin{equation}\label{eq:conv_quad_var}
    \frac{\langle M \rangle_n}{n} \convas 1 - \bigg(1 - \frac{2}{d}\bigg)^2 = \frac{4(d - 1)}{d^2}.
\end{equation}
Further, since $|Y_j| \le 4$, we have that for any $\delta > 0$, the Lyapunov condition
\[
    \frac{1}{(\sqrt{n})^{2 + \delta}}\sum_{k = 1}^n \bbE|Y_k|^{2 + \delta} = o(1),
\]
is satisfied. As such, by the martingale central limit theorem, it follows that
\[
    \sqrt{n} \bigg(\frac{\Delta_n}{n} - \bigg[ 1 - \frac{2(1 - p)}{d - 1} \cdot \frac{1}{n} \sum_{k = 0}^{n - 1} \ind(\Delta_k > 0) + \frac{2(1 - pd)}{d - 1} \cdot \frac{1}{n}\sum_{k = 0}^{n - 1} \frac{\ell_k}{k} \bigg] \bigg) \convd N\bigg(0, \frac{4(d - 1)}{d^2}\bigg).
\]
Also, since
\[
    \frac{1}{n} \sum_{k = 0}^{n - 1} \ind(\Delta_k = 0) = O_P\bigg(\frac{1}{n}\bigg),
\]
we may rewrite the above result as
\[
    \sqrt{n}\bigg[\frac{\Delta_n}{n} - \frac{d - 2}{d} - \frac{2(1 - pd)}{d - 1} \cdot \Xi_n \bigg] \convd N\bigg(0, \frac{4(d - 1)}{d^2}\bigg).
\]
This completes the proof.
\end{proof}

\section{Conclusion}\label{sec:conc}
In conclusion, we have introduced a generalisation of the elephant random walk of \cite{schutz2004elephants} to finite generated groups and their Cayley graphs. Focusing on the Bethe lattice, we have shown, by leveraging the connection of the ERW dynamics with urn models first observed in \cite{baur2016elephant}, that memory has no effect on the first order behaviour of the asymptotic speed, thus matching speed of the simple random walk. On the other hand, the rate of convergence to the limiting speed appears to undergo a phase transition, as evidenced by our upper bounds and numerical experiments. Obtaining precise limit distributions is an interesting open question, as is proving exponential decay of the return probability in all memory regimes. Answering these questions would potentially require a more careful decoupling of the underlying urn process governing the step counts from the geometry of the Cayley graph. Investigation of how memory interacts with geometry in other finitely generated infinite groups is another intriguing question, which we hope to return to in the future.

\section*{Acknowledgements}
The author was partially supported by the Prime Minister Early Career Research Grant ANRF/ECRG/2024/006704/PMS from the Anusandhan National Research Foundation, Govt.~of India.

\bibliographystyle{alpha}
\bibliography{refs.bib}

\end{document}